\documentclass[11pt,reqno]{amsart}
\newcommand{\mysection}[1]{\section{#1}
      \setcounter{equation}{0}}

\newcommand\cbrk{\text{$]$\kern-.15em$]$}} 
\newcommand\opar{\text{\raise.2ex\hbox{${\scriptstyle | }$}\kern-.34em$($} }

\usepackage{color}
\usepackage{amsmath,amsthm,amssymb,amsfonts,enumerate,color,enumerate}
\usepackage[pdftex]{graphicx}

\usepackage{verbatim}
\usepackage[spanish,USenglish]{babel} 
\usepackage{color}
\usepackage{tikz}
\allowdisplaybreaks

\DeclareMathOperator*{\esssup}{ess\,sup}

\newtheorem{theorem}{Theorem}[section]
\newtheorem{lemma}[theorem]{Lemma}
\newtheorem{proposition}[theorem]{Proposition}
\newtheorem{corollary}[theorem]{Corollary}

\theoremstyle{definition}
\newtheorem{assumption}{Assumption}[section]
\newtheorem{definition}{Definition}[section]

\theoremstyle{remark}
\newtheorem{remark}{Remark}[section]

\newcommand{\R}{\mathbb{{R}}}

\newcommand{\N}{\mathbb{{N}}}

\newcommand\bF{\mathbb{F}}

\newcommand\bI{\mathbb{I}}

\newcommand\bL{\mathbb{L}}
\newcommand\bR{\mathbb{R}}

\newcommand\bS{\mathbb{S}}

\newcommand\bW{\mathbb{W}}

\newcommand\frJ{\mathfrak{J}}

\newcommand\cA{\mathcal{A}}
\newcommand\cB{\mathcal{B}}

\newcommand\cH{\mathcal{H}}

\newcommand\cJ{\mathcal{J}}

\newcommand\cL{\mathcal{L}}
\newcommand\cM{\mathcal{M}}
\newcommand\cN{\mathcal{N}}

\newcommand\cR{\mathcal{R}}
\newcommand\cS{\mathcal{S}}
\newcommand\cU{\mathcal{U}}
\newcommand\cZ{\mathcal{Z}}

\makeatletter
 \newcommand{\sumstar}
 {\operatornamewithlimits{\sum@\kern-.2em\raise1ex\hbox{*}}}
 \makeatother

\begin{document}

\author[M. De Le\'on-Contreras]{Marta De Le\'on-Contreras}
\address{Departamento de Matem\'aticas, 
Facultad de Ciencias, Universidad Aut\'onoma de Madrid, Spain.}
\email{marta.leon@uam.es}

\author[I. Gy\"ongy]{Istv\'an Gy\"ongy}
\address{School of Mathematics and Maxwell Institute, University of Edinburgh, Scotland, United Kingdom.}
\email{i.gyongy@ed.ac.uk}

\author[S. Wu]{Sizhou Wu}
\address{School of Mathematics,
University of Edinburgh,
King's  Buildings,
Edinburgh, EH9 3JZ, United Kingdom}
\email{Sizhou.Wu@ed.ac.uk}

\keywords{Integro-differential equation, Bessel potential spaces, Interpolation couples}

\subjclass[2010]{Primary 45K05, 35R09; Secondary 47G20}

\begin{abstract} A class of (possibly) degenerate integro-differential equations of parabolic type is considered, 
which includes the Kolmogorov equations for jump diffusions. 
Existence and uniqueness of the solutions are established in Bessel potential spaces 
and in Sobolev-Slobodeckij spaces. Generalisations 
to stochastic integro-differential equations, 
arising in filtering theory of jump diffusions, will be given in a forthcoming paper. 
\end{abstract}

\title[Integro-Differential Equations]{On solvability 
of integro-differential equations}

\maketitle

\mysection{Introduction}
We consider the equation 
\begin{equation}                                                                     \label{eq1}
\frac{\partial}{\partial t} u(t,x)=\mathcal{A} u(t,x)+f(t,x)
\end{equation}
on $H_T=[0,T]\times\bR^d$ for a given $T>0$, 
with initial condition $u(0,x)=\psi(x)$ for $x\in\R^d$, 
where $\cA$ is an integro-differential operator 
of the form $\cA=\cL+\cM+\cN+\cR$, with a ``zero-order" linear operator 
$\cR$, a second order differential operator 
$$
\mathcal{L}(t)=a^{ij}(t,x)D_{ij}+b^i(t,x)D_i +c(t,x) 
$$
and linear operators  $\cM$ and $\cN$ defined by 
\begin{equation}                                             \label{def M}
\cM(t) \varphi(x)=\int_{Z}(\varphi(x+
\eta_{t,z}(x))-\varphi(x)-\eta_{t,z}(x)\nabla \varphi(x))\,\mu(dz), 
\end{equation}
\begin{equation}                                             \label{def N}
\mathcal{N}(t)\varphi(x)=\int_{Z}(\varphi(x+\xi_{t,z}(x))-\varphi(x))\,{\nu}(dz)
\end{equation}
for a suitable class of real-valued functions 
$\varphi(x)$ on $\bR^d$. Here $a^{ij}$, $b^i$ and $c$ 
are real-valued bounded functions defined on $H_T$, $\mu$ 
and $\nu$ are $\sigma$-finite measures on a measurable space $(Z,\cZ)$. 
The functions  
$\eta$ and $\xi$ are $\R^d$-valued mapping defined on $H_T\times Z$. 
Under ``zero-order operators" we mean bounded linear operators $\cR$ mapping 
the Sobolev spaces $W^k_p$ into themselves for $k=0,1,2,..,n$ for some $n$.  
Examples include integral operators $\cR(t)$ defined by 
\begin{equation}                                                               \label{def R}
\cR(t)\varphi(x)=\int_{Z}\varphi(x+\zeta_{t,z}(x))\,{\lambda}(dz)
\end{equation}
with appropriate functions $\zeta$ on $H_T\times Z$ and finite measures $\lambda$ 
on $\cZ$. 

Our aim is to investigate the solvability of equation \eqref{eq1} 
in Bessel potential spaces $H^m_p$ and Sobolev-Slobodeckij  spaces $W^m_p$ 
for $p\geq2$ and $m\in[1,\infty)$. 

Such kind of equations arise, for example, as Kolmogorov equations for Markov 
processes given by stochastic differential equations, driven by Wiener processes 
and Poisson random measures, see e.g., \cite{A}, \cite{BMR}, \cite{GM}, \cite{GM2} and \cite{Ja}. 
They play important roles in studying 
random phenomena modelled by 
Markov processes with jumps, in physics, biology, engineering and finance, 
see e.g., \cite{BV}, \cite{CT}, \cite{Mu}, \cite{Pa} 
and the references therein.  There is a huge literature on the solvability of 
these equations, but in most of the publications some kind of non-degeneracy,  
conditions on the equations, or specific assumptions on the measures $\mu$ 
and $\nu$ are assumed. Results in this direction 
can be found, for example,  
in \cite{GM}, \cite{GM2}, \cite{Ja},  \cite{MPh1}, \cite{MP1}, \cite{MP2}, 
 \cite{MX} and \cite{Pr1}, 
and for nonlinear equations of the type \eqref{eq1}, arising in the theory of stochastic control 
of random processes with jumps, we refer to \cite{GM2} and \cite{Pr2}. 
Extensions of the $L_p$-theory of Krylov \cite{Kr1} to stochastic equations and 
systems 
of stochastic equations  
with integral operators of the type 
$\cM$ and $\cN$ above are developed in \cite{CheKim}, \cite{CL}, \cite{KK}, 
\cite{KL} and \cite{MP3}. 

Note that, since with a positive constant $c_{\alpha,d}$ 
the fractional Laplacian operator $\Delta^{\alpha/2}:=-(-\Delta)^{\alpha/2}$ 
 has the integral representation
$$
\Delta^{\alpha/2}\varphi(x)=
\lim_{\varepsilon\downarrow0}c_{\alpha,d}
\int_{|z|\geq\varepsilon}(\varphi(x+z)-\varphi(x))\frac{1}{|z|^{d+\alpha}}\,dz,  
\quad \alpha\in(0,2) 
$$
for smooth functions $\varphi$ with compact support on $\bR^d$, 
we have 
$\Delta^{\alpha/2}=\cN+\cR-\bar c_{\alpha,d}$ for $\alpha\in(0,1)$ and 
$\Delta^{\alpha/2}=\cM+\cR-\bar c_{\alpha,d}$ for $\alpha\in[1,2)$, where 
$\cM$, $\cN$ and $\cR$ are defined in \eqref{def M}, \eqref{def N} 
and \eqref{def R}, with 
$\eta_{t,z}(x)=\xi_{t,z}(x)=\zeta_{t,z}(x)=z\in Z:=\bR^d\setminus\{0\}$,   
$\mu(dz)=\nu(dz)=c_{\alpha,d}{\bf 1}_{|z|\leq1}|z|^{-d-\alpha}\,dz$, 
$\lambda(dz)=c_{\alpha,d}{\bf 1}_{|z|>1}|z|^{-d-\alpha}\,dz$ and with
$$
\bar c_{\alpha,d}=c_{\alpha,d}\int_{|z|>1}\frac{1}{|z|^{d+\alpha}}\,dz.
$$ 
Thus examples for equation  \eqref{eq1} include equations with 
$\Delta^{\alpha/2}$, 
$\alpha\in(0,2)$. There are many important results in the literature about 
fractional operators and about equations containing them, see e.g., 
\cite{CS}, \cite{CSS}, \cite{ST} and the references therein. 

In this paper we are interested in the solvability of equation \eqref{eq1} when 
it  can degenerate, and besides some integrability conditions,  no 
specific conditions on the measures $\mu$ and $\nu$ are assumed.  
An $L_2$-theory of degenerate linear elliptic and parabolic PDEs is developed
in \cite{O1}, \cite{O2}, \cite{OR1} and 
\cite{OR2}. The solvability in $L_2$-spaces of linear degenerate
stochastic PDEs of parabolic type were first studied in \cite{KR1977} (see also \cite{Ro}). 
The first existence and uniqueness theorem on solvability of these
equations in $W^m_p$ spaces, for integers $m\geq1$ and any $p\geq2$, 
is presented in \cite{KR}.
A gap in the proof of a crucial $L_p$-estimate in   
\cite{KR} is filled in, and the existence and uniqueness theorem is substantially  
improved in \cite{GGK}. 
The solvability of degenerate stochastic integro-differential equations, which include 
the type of equations \eqref{eq1}, are studied in \cite{Da}, \cite{LM1} and \cite{LM2}.  
Existence and uniqueness theorems are obtained  in H\"older spaces in 
\cite{LM1}, and in $L_2$-spaces 
 in \cite{Da} and \cite{LM2}. 
Our main result,  Theorem \ref{theorem main} below, is an existence 
and uniqueness theorem in $L_p$-spaces,  
which generalises the corresponding results in \cite{Da} and \cite{LM2}, 
but instead of stochastic integro-differential equations here we consider only 
the deterministic equation \eqref{eq1}.  
A generalisation of Theorem \ref{theorem main} to stochastic integro-differential 
equations will be presented in a forthcoming paper.

In conclusion we introduce some notations used throughout the paper. 
For vectors $v=(v^i)$ and $w=(w^i)$ in $\bR^d$ we 
use the notation $vw=\sum_{i=1}^mv^iw^i$ and $|v|^2=\sum_i|v^i|^2$. 
For real-valued Lebesgue measurable functions 
$f$ and $g$ defined on $\bR^d$ the notation $(f,g)$ 
means the integral of the product $fg$ over $\bR^d$ 
 with respect to the Lebesgue measure on $\bR^d$.   
A finite list $\alpha=\alpha_1\alpha_2,...,\alpha_n$  of numbers $\alpha_i\in\{1,2,...,d\}$ 
is called a multi-number of length $|\alpha|:=n$, and the notation 
$$
D_{\alpha}:=D_{\alpha_1}D_{\alpha_2}...D_{\alpha_n}
$$ 
is used for integers $n\geq1$, 
where 
$$
D_i=\frac{\partial}{\partial x^i}, \quad \text{for $i\in\{1,2,...,d\}$}. 
$$
We use also the multi-number $\epsilon$ of length $0$, 
and agree that $D_{\epsilon}$ means the identity operator. 
For an integer $n\geq0$ and functions $v$ on $\bR^d$, whose 
partial derivatives up to order $n$ are functions, 
we use the notation $D^nv$ for the collection 
$
\{D_{\alpha}v:|\alpha|=n\}, 
$
and define
$$
|D^nv|^2=\sum_{|\alpha|=n}|D_{\alpha}v|^2. 
$$
For differentiable functions $v=(v^1,...,v^d):\bR^d\to \bR^d$ 
the notation $Dv$ means the Jacobian matrix 
whose $j$-th entry in the $i$-th row is $D_jv^i$. 

For a separable 
Banach space $V$ we use the notation $L_p([0,T],V)$ for the space of 
Borel functions $f:[0,T]\to V$ such that $|f|^p_V$ has finite integral with respect 
to the Lebesgue measure on $[0,T]$. The Borel $\sigma$-algebra on $V$ 
is denoted by $\cB(V)$. 
 The notations $C([0,T],V)$ 
and $C_w([0,T],V)$ mean the space of $V$-valued functions on $[0,T]$, 
which are continuos with respect to the strong topology and with respect to the 
weak topology, respectively, on $V$. 
For $m\in\bR$ and $p\in(1,\infty)$ we use the notation $H^m_p$ 
for the Bessel potential space 
with exponent $p$ and order $m$, defined as the space of 
generalised functions  $\varphi$ on $\bR^d$ such that 
$$
(1-\Delta)^{m/2}\varphi\in L_p\quad\text{and}\quad 
|\varphi|_{H^m_p}:=|(1-\Delta)^{m/2}\varphi|_{L_p}<\infty,  
$$ 
where $\Delta=\sum_{i=1}^dD_i^2$, and $L_p$ is the space of real-valued 
Borel functions $f$ on $\bR^d$ such that 
$$
|f|^p_{L_p}:=\int_{\bR^d}|f(x)|^p\,dx<\infty. 
$$
For $p\in[1,\infty)$ and integers $m\geq0$ the notation $W^m_p$ means the 
Sobolev space defined as the completion of $C_0^{\infty}$, the space of smooth functions 
with compact support on $\bR^d$, in the norm 
$$
|\varphi|_{W^m_p}:=\sum_{|\alpha|\leq m}|D_{\alpha}\varphi|_{L_p}.  
$$
For integers $m\geq0$ the space $W^{m}_{\infty}$ 
is the completion of $C^{\infty}_b$, 
the space of bounded functions on $\bR^d$ with bounded smooth derivatives,  
in the norm 
$$
|\varphi|_{W^m_{\infty}}:=\sum_{|\alpha|\leq m}\esssup|D_{\alpha}\varphi|. 
$$
One knows that $H^m_p$ and 
$W^m_p$ are the same as vector spaces, and their norms are equivalent for 
$p\in(1,\infty)$ and integers $m\geq0$. 
When $m>0$ is not an integer, then $W^m_p$ denotes 
space of functions $f\in W^{\lfloor m\rfloor}_p$ such that 
$$
[D_{\alpha}f]^p_{\{m\},p}:=\int_{\bR^d}\int_{\bR^d}
\frac{|D_{\alpha}f(x)-D_{\alpha}f(y)|^p}{|x-y|^{p\{m\}+d}}\,dx\,dy<\infty 
$$
for every multi-index $\alpha$ of length $\lfloor m\rfloor$, 
where $\lfloor m\rfloor$ is the largest integer smaller than  $m$, and 
$\{m\}=m-\lfloor m\rfloor$. When $m>0$ is not an integer, then $W^m_p$ with the norm 
$$
|f|_{W^m_p}=|f|_{W^{\lfloor m\rfloor}_p}+\sum_{|\alpha|=\lfloor m\rfloor}[D_{\alpha}f]_{\{m\},p}
$$
 is a Banach space, called {\it Slobodeckij space}. 
Derivatives are understood in the generalised sense unless otherwise noted.
The summation convention with respect to repeated indices 
is used thorough the paper, where it is not indicated otherwise. For basic notions 
and results on solvability of parabolic PDEs 
in Sobolev spaces we refer to \cite{Kbook}. 

The paper is organised as follows. 
The formulation of the problem and the main result, Theorem \ref{theorem main}, 
is in Section \ref{section formulation}. Some technical tools and the crucial 
$L_p$ estimates are collected in Sections \ref{section preliminaries} 
and \ref{section estimates}, respectively. The proof of Theorem \ref{theorem main} 
is given in the last section, Section \ref{section mainproof}.

\mysection{Formulation of the main results}                 \label{section formulation}

Let $K$ be a constant and let $\bar\eta$ and $\bar\xi$ be nonnegative 
$\cZ$-measurable functions on $Z$ such that 
$$
\quad K^2_{\eta}:=\int_{Z}\bar\eta^2(z)\,\mu(dz)<\infty,
 \quad K_{\xi}:=\int_{Z}\bar{\xi}(z)\,\nu(dz)<\infty. 
$$
Let $p\in[2,\infty)$ and $m\geq0$ be real numbers, and let 
$\lceil m\rceil$ denote  the smallest integer which is greater than or equal to $m$. 
We make the following assumptions.  
\begin{assumption}                                                                           \label{assumption L}
The derivatives of $c$ in $x\in\bR^d$ 
up to order $\lceil m\rceil$, and the derivatives of $b^i$ in $x$ up to 
order $\max\{\lceil m\rceil,1\}$  are Borel functions on $H_T$, bounded by 
$K$ for all $i=1,2,..,d$. The derivatives of $a^{ij}$ in $x$ 
up to order $\max\{\lceil m\rceil,2\}$ are  Borel functions 
on $H_T$ for $i,j=1,...,d$, and are bounded by $K$. 
Moreover, 
$a^{ij}=a^{ji}$ for all $i,j=1,...,d$ and for $dt\otimes dx$-almost all $(t,x)\in H_T$ 
\begin{equation}                                            \label{a}
a^{ij}z^iz^j\geq0\quad\text{for all $(z^1,...,z^d)\in\bR^d$}. 
\end{equation}
\end{assumption}

\begin{assumption}                                                                           \label{assumption M}
The function $\eta=(\eta^i)$ is an $\bR^d$-valued $\cB(H_T)\otimes\cZ$-measurable 
mapping on $H_T\times Z$, its derivatives in $x\in\bR^d$ 
up to order $\max\{\lceil m\rceil,3\}$ exist and are continuous in $x$,  such that 
$$
|\eta|\leq\bar \eta,\quad |D^k\eta|\leq \bar\eta\wedge K, \quad k=1,2,...,\max(\lceil m\rceil,3)=:m_{\eta}
$$
for all $(t,x,z)\in H_T\times Z$, and  
$$
K^{-1}\le |\det(\mathbb{I}+\theta D \eta_{t,z}(x))|
$$
for all $(t,x,z,\theta)\in H_T\times Z\times [0,1]$, 
where $\mathbb{I}$ is the $d\times d$ identity matrix and recall that 
$D \eta$ denotes the Jacobian matrix of $\eta$. 
\end{assumption}
\begin{remark}                                                                    \label{remark ibp}
By Taylor's formula we have 
$$
v(x+\eta(x))-v(x)-\eta(x)\nabla v(x)
=\int_0^1\eta^k(x)(v_k(x+\theta\eta(x))-v_k(x))\,d\theta
$$
$$
=\int_0^1\eta^k(x)D_k(v(x+\theta\eta(x))-v(x))\,d\theta
-\int_0^1\theta\eta^k(x)\eta^l_k(x)v_l(x+\theta\eta(x))\,d\theta
$$
for every $v\in C_0^{\infty}$, where to ease notation 
we do not write the arguments $t$ and $z$ and write $v_k$ instead of $D_kv$ 
for functions $v$. 
Due to Assumption \ref{assumption M} these equations extend to $v\in W^1_p$ 
for $p\geq2$ as well. 
Hence after changing the order of integrals, by integration by parts we 
obtain 
$$
(\cM v,\varphi)=-(\cJ^{k}v,D_k\varphi)+(\cJ^{0}v,\varphi)
$$
for $\varphi\in C_0^{\infty}$, 
with 
\begin{align}                                                       
\cJ^{k}(t)v(x)
=&\int_0^1\int_{Z}
\eta^k(v(\tau_{\theta\eta}(x))-v(x))\,\mu(dz)\,d\theta, \quad k=1,2,...,d,      \label{def Jk}\\
\cJ^{0}(t)v(x)
=&-\int_0^1\int_{Z}
\{\sum_k\eta^k_k(v(\tau_{\theta\eta}(x))-v(x))+\theta\eta^k(x)\eta^l_k(x)v_l(\tau_{\theta\eta}(x))\}
\,\mu(dz)\,d\theta,                                                                                                          \label{def J0}
\end{align}
where for the sake of short notation   
the arguments $t,z$ of $\eta$ and $\eta_k$ have been omitted, and 
\begin{equation}                                                                                                       \label{tau0}
\tau_{\theta\eta}(x):=x+\theta\eta_{t,z}(x)\quad \text{for $x\in\bR^d$, $t\in[0,T]$, $z\in Z$ and 
$\theta\in[0,1]$}.  
\end{equation}

\end{remark}

\begin{assumption}                                                            \label{assumption N}
The function $\xi=(\xi^i)$ is an $\bR^d$-valued $\cB(H_T)\otimes\cZ$-measurable 
mapping on $H_T\times Z$, its derivatives in $x\in\bR^d$ 
up to order $\max\{\lceil m\rceil,2\}$ exist and are continuous in $x$ such that 
$$
|\xi|\leq\bar \xi,\quad |D^k\xi|\leq \bar\xi\wedge K, \quad k=1,2,...,\max(\lceil m\rceil,2)=:m_{\xi}
$$
for all $(t,x,z)\in H_T\times Z$, 
and 
$$
K^{-1}\le |\det(\mathbb{I}+\theta D \xi_{t,z}(x))|
$$
for all $(t,x,z,\theta)\in H_T\times Z\times [0,1]$.  
\end{assumption}

\begin{assumption}                                                            \label{assumption R}
For $\varphi\in C_0^{\infty}$ we have 
$$
|\cR\varphi|_{W^n_p} \leq K|\varphi|_{W^n_p}\quad 
\text{for integers $n=0,1,...,\lceil m\rceil $}. 
$$
\end{assumption}
\begin{remark}
Obviously there are many important examples of linear operators 
satisfying this condition. By Lemma \ref{lemma chainrule2} below it is not difficult to show that 
the operator $\cR$ defined in \eqref{def R} satisfies 
Assumption \ref{assumption R} if 
$\zeta=(\zeta^i)$ is an $\bR^d$-valued $\cB(H_T)\otimes\cZ$-measurable 
mapping on $H_T\times Z$ and it is a $C^{\lceil m\rceil}$-diffeomorphism 
of $\bR^d$ for every $(t,z)\in[0,T]\times Z$
such that 
$$
|D^k\zeta|\leq K, \quad k=1,2,...,\lceil m\rceil,\quad 
K^{-1}\le |\det(\mathbb{I}+D \zeta_{t,z}(x))|
$$
for all $(t,x,z)\in H_T\times Z$.
\end{remark}

Let $V^s_p$ denote $H^s_p$ or $W^s_p$ for every $s\geq0$. 
\begin{assumption}                                                                         \label{assumption free} 
We have $\psi\in V^m_p$ and $f\in L_p([0,T],V^m_p)$. 
\end{assumption}

Using Remark \ref{remark ibp} 
we define the notion of generalised solutions to \eqref{eq1} 
as follows. 
 
\begin{definition}                                                                           \label{definition solution}
An $L_p(\bR^d)$-valued continuous function 
$u=u(t)$, $t\in[0,T]$ is a generalised solution 
to equation \eqref{eq1} with initial condition 
$u(0)=\psi$, if $u(t)\in W^1_p(\bR^d)$ 
for $dt$-almost every $t\in[0,T]$, $u\in L_p([0,T], W^1_p)$,  
and	
\begin{equation}                                                                      \label{solution}
	(u(t),\varphi)=(\psi,\varphi)
	+\int_0^t\langle\cA u(s),\varphi\rangle+(f(s),\varphi)\,ds 
\end{equation}
for every $\varphi\in C_0^{\infty}(\bR^d)$ and $t\in[0,T]$, where 
$$
\langle\cA u,\varphi\rangle
:=-(a^{ij}D_ju, D_i\varphi)+(\bar {b}^iD_iu+cu,\varphi)
-(\cJ^{i}u,D_i\varphi)+(\cJ^{0}u,\varphi)
$$
$$
+(\cN u,\varphi)+(\cR u,\varphi)
$$
with $\bar b^i=b^i-D_ja^{ij}$.
 \end{definition}
 
Observe that, if Assumptions \ref{assumption M} and \ref{assumption N} hold, 
then there is a constant $N$ such that 
$$
|\mathcal{J}^{(0)}(s)v|_{L^p} \le N|v|_{W^1_p},
\quad
|\mathcal{J}^{(k)}(s)v|_{L^p} \le N|v|_{W^1_p}, 
\quad 
|\mathcal{N}(s)v|_{L^p} \le N|v|_{W^1_p}, 
$$
for all $v\in W^1_p$ and $s\in[0,T]$ 
(see Proposition \ref{proposition LMN} below).  
Thus 
$\langle\cA u,\varphi\rangle$ is well-defined when 
Assumptions \ref{assumption L} through \ref{assumption R} are satisfied.

\begin{theorem}                                                                                     \label{theorem main}
Let Assumptions 
\ref{assumption L} through \ref{assumption free} hold with $m\geq1$. 
Then equation \eqref{eq1} with initial condition $u(0)=\psi$ 
has a generalised solution $u$, which is a weakly continuous 
$V^{m}_p$-valued function, 
and it is strongly continuous as a $V^{s}_p$-valued function of $t\in[0,T]$ for any $s<m$.  
Moreover, there is a constant $N=N(K,d,m,p,T,K_{\xi},K_{\eta})$ such that
\begin{equation}                                                                                                \label{estimatesup}
	\sup_{t\le T}|u(t)|_{V^s_p}^p
	\leq N \left(|\psi|_{V^s_p}^p+\int_0^T|f(t)|_{V^s_p}^p dt \right) \quad\text{for $s\in[0,m]$}.  
	\end{equation}
If Assumptions 
\ref{assumption L} through \ref{assumption free} hold with $m=0$, then there is at most one 
generalised solution.  
\end{theorem}

\mysection{preliminaries}                                                                          \label{section preliminaries}

First we present some lemmas which are probably well-known 
from textbooks in analysis. 
Recall that we use 
multi-numbers 
$\alpha=\alpha_1\dots\alpha_n$, where  $\alpha_j\in\{1,\dots, d\}$, 
to denote higher order derivatives.  
For a multi-number $\alpha=\alpha_1....\alpha_k$ 
of length $k$ and a subset $\kappa$ of $\bar k:=\{1,2,...,k\}$ 
we use the notation $\alpha(\kappa)$ for 
the multi-number $\alpha_{l_1}...\alpha_{l_n}$, 
where $l_1$,...,$l_n$ are the elements 
of $\kappa$, listed in increasing order. 
For short we use the 
notation $v_{\alpha}:=D_{\alpha}v$ for functions $v$ of $x\in\bR^d$. 
We write $\kappa_1\sqcup\dots\sqcup\kappa_n=\bar k$ 
for  the partition of $\bar k:=\{1,2,..,k\}$ 
into $n$ nonempty disjoint sets 
$\kappa_1$,...,$\kappa_n$. Two partitions are considered different 
if one of the sets  in one of the partitions 
is different from each set in the other partition. 
Using the above notation the chain rule for 
$(u(\rho))_{\alpha}:=D_{\alpha}(u(\rho))$ for functions $u:\bR^d\to\bR$ 
and $\rho:\bR^d\to\bR^d$ 
can be formulated as follows. 

\begin{lemma}                                                                    \label{lemma chainrule} 
Assume that the derivatives of $u$ and $\rho=(\rho^1,...\rho^d)$ up to order $k\geq1$ 
exist and are continuous functions. 
Then for any multi-number 
$\alpha=\alpha_1\alpha_2...\alpha_l$ of length $l\in\{1,2,...,k\}$
we have 
\begin{equation}                                                                                        \label{chainrule}
(u(\rho))_\alpha=\sum_{n=1}^l\sum_{\kappa_1\sqcup\dots\sqcup\kappa_n
=\bar l}u_{i_1\dots i_n}(\rho)\rho_{\alpha(\kappa_1)}^{i_1}
\rho_{\alpha(\kappa_2)}^{i_2}\dots \rho_{\alpha(\kappa_n)}^{i_n},
\end{equation}
where the second summation on the right-hand side 
means summation over the different partitions 
of $\bar l:=\{1,2,...,l\}$, and for each $l$ and each partition of $\bar l$ there is also a summation 
with respect to the repeated indices $i_j\in\{1,2,...,d\}$ for 
$j=1,2,...,n$. 
\end{lemma}

\begin{proof}
One can prove this lemma by induction on $l$, and it is left for the reader as an easy exercise. 
\end{proof} 

A one-to-one function, mapping $\bR^d$ onto $\bR^d$, 
is called a $C^k(\bR^d)$-diffeomorphism 
on $\bR^d$ for an integer $k\geq1$, if the derivatives up to order $k$ of the function 
and its inverse are continuous. If $\rho$ is a 
$C^k(\bR^d)$ diffeomorphism such that 
\begin{equation}                                                                    \label{diffeo korlat}                                                      
M\leq |\det(D\rho)| 
\,\,
{\text{and\,\,  $|D^i\rho|\leq N$ for $i=1,2,...,k$}} .                                                      
\end{equation}
for some positive constants $M$ and $N$, then Lemma \ref{lemma chainrule} 
can be extended to $u\in W^k_p$ for any $p\in[1,\infty)$. 

\begin{lemma}                                             \label{lemma chainrule2}
Let $\rho$ be a $C^k(\bR^d)$-diffeomorphism for some $k\geq1$ 
such that \eqref{diffeo korlat} holds. 
Then the following statements hold. 
\begin{enumerate}[(i)]
\item There is a constant $C=C(M,N,d,p,k)$ such that for 
$u\in W^l_p$, $p\in[1,\infty]$ and $v\in W^l_{\infty}$  
\begin{equation}                                                    \label{rule}                     
|u(\rho)v|_{W^l_p}\leq C|u|_{W^l_p}|v|_{W^l_\infty}
\end{equation}
for $l=0,1,2,...,k$. 
\item For $1\leq |\alpha|\leq k$ equation \eqref{chainrule} holds $dx$-almost 
everywhere for any $u\in W^k_p$, $p\in[1,\infty]$.  
\end{enumerate}
\end{lemma}
\begin{proof}
We prove  \eqref{rule} by induction on $l$, assuming  
that $u\in W^k_p$, $v\in W^k_{\infty}$ 
are smooth functions and $p\neq\infty$.  
For $l=0$ by the change of 
variable $\rho(x)=y$ and by the first inequality in \eqref{diffeo korlat} we have 
$$
|u(\rho)v|_{L_p}^p\leq\esssup |v|^p\int_{\bR^d}|u(y)|^p|{\rm{det}}D\rho^{-1}(y)|\,dy
$$
$$
=\esssup|v|^p\int_{\bR^d}|u(y)|^p|{\rm{det}}D\rho(\rho^{-1}(y))|^{-1}\,dy
\leq M^{-1}|u|_{L_p}^p\esssup|v|^p, 
$$
which proves \eqref{rule} for $l=0$. Let $l\geq1$ and assume that 
statement (i) is true for $l-1$ in place of $l$. By the Leibniz rule and the 
chain rule 
$$
D_i(u(\rho)v)=u_j(\rho)\rho^j_iv+u(\rho)v_i\quad\text{for each $i=1,2,...,d$}.
$$
Hence by the induction hypothesis and the second inequality in  \eqref{diffeo korlat} 
we have 
$$
|D_i(u(\rho)v)|_{W^{l-1}_p}\leq |u_j(\rho)\rho^j_iv|_{W^{l-1}_p}
+|u(\rho)v_i|_{W^{l-1}_p}
$$
$$
\leq C|u_j|_{W^{l-1}_p}|\rho^j_iv|_{W^{l-1}_{\infty}}
+C|u|_{W^{l-1}_p}|v_i|_{W^{l-1}_{\infty}}\leq C(Nd+1)|u|_{W^l_p}|v|_{W^l_{\infty}}. 
$$
Thus 
$$
|u(\rho)v|_{W^{l}_p}=\sum_{i=1}^d|D_i(u(\rho)v)|_{W^{l-1}_p}
\leq Cd(Nd+1)|u|_{W^l_p}|v|_{W^l_{\infty}}, 
$$
which finishes the induction proof. When $p=\infty$ and $l=0$ then 
\eqref{rule} is obvious, and by induction on $l$ we get the result as before. 
Clearly, the condition given by the first inequality in \eqref{diffeo korlat}  
is not needed in this case. 
Since $C_0^{\infty}$ is dense in $W^l_p$ when $p\neq\infty$ and 
$C_b^{\infty}$ is dense in $W^l_p$, we can finish the proof of (ii) by a standard 
approximation argument. Making use of (ii) we can get (i) also by approximating 
$u$ by $C_0^{\infty}$ functions when $p\neq\infty$ and by $C^{\infty}_b$ 
functions when $p=\infty$. 
\end{proof}

 \begin{lemma}                                                                      \label{lemma diff}
Let $\rho$ be 
a $C^k(\bR^d)$-diffeomorphism for $k\geq1$, such that 
\eqref{diffeo korlat} holds.
Then
there are positive constants $M'=M'(N,d)$ and $N'=N'(N,M,d,k)$ 
such that \eqref{diffeo korlat} holds 
with $g:=\rho^{-1}$, the inverse of $\rho$, in place of $\rho$, 
with $M'$ and $N'$ in place of $M$ 
and $N$, respectively.
\end{lemma}

\begin{proof}
It follows from the second estimate in \eqref{diffeo korlat} 
that $|\det(D\rho)|\leq d!N^d$, 
and since $Dg(x)=(D\rho)^{-1}(g(x))$, we have 
$$
|\det Dg(x)|=|\det (D\rho)(g(x))|^{-1}\geq (d!N^d)^{-1}, 
$$
which proves the first estimate in \eqref{diffeo korlat} 
for $g=\rho^{-1}$ in place of $\rho$. 
To estimate $|Dg|$ notice that $\|Dg(x)\|=\lambda_{1}$,  
where $\|Dg(x)\|$ is the operator norm of the matrix $Dg(x)$, and 
$\lambda_1\geq\lambda_2\geq...\geq\lambda_d>0$ 
are the singular values of the matrix $Dg(x)$. 
Since $1/\lambda_d\geq1/\lambda_{d-1}\geq...\geq1/\lambda_{1}$ are 
the singular values of  $A(x):=(D\rho)(g(x))$, 
we have $|\det A(x)|=1/\Pi_{i=1}^d\lambda_i\geq M$ 
and $\|A(x)\|=1/\lambda_{d}\leq N$.  
Hence
\begin{equation}                                                                                \label{D}      
|D\rho^{-1}(x)|\leq K_0\|D\rho^{-1}(x)\|
=K_0\lambda_{1}\leq K_0(N\lambda_{d})^{d-1}\lambda_{1}
\leq K_0N^{d-1}\prod_{i=1}^d\lambda_i
\leq\frac{K_0N^{d-1}}{M}   
\end{equation}
with a constant $K_0=K_0(d)$. 
To estimate $|D^ig|$ for $1\leq i\leq k$ and $k>1$,  
we claim that for every multi-number $\alpha$ of length $i<k$ 
 each entry $B^{rl}(\alpha)$ of the matrix $B({\alpha}):=D_{\alpha}Dg$
 is a linear 
combination of products of at most $k+2$ functions, 
with multiplicity, taken from the set 
$$
\{\rho^j_{\beta}(g), \,g^r_{\gamma}: j,r=1,2,..,d,\,1\leq|\beta|\leq k,\, 1\leq|\gamma|<k\}
$$
with integer coefficients, determined by $\alpha$ and $d$, where $v_{\beta}:=D_{\beta}v$ 
for functions $v$ and multi-numbers $\beta$.  
By the chain rule from $\rho(g(x))=x$ we have $ADg=I$ with $A=(D\rho)(g)$. 
Hence, for $|\alpha|=1$
$$
D_{\alpha}Dg=-A^{-1}D_{\alpha}ADg=-DgD_{\alpha}ADg=:B(\alpha).  
$$
This gives 
$B^{rl}(\alpha)=-g^r_j\rho^j_{pi}(g)g^i_{\alpha}g^p_l$ for $r,l=1,2,..,d$,  
which proves the claim for $k=2$, and our claim follows by induction 
on $k$. Hence also by induction on $k$ we immediately obtain that 
 \begin{equation*}                                                   
 |D^ig|\leq N' \quad\text{for $1\leq i\leq k$ with a constant $N'=N'(N,M,d,k)$}, 
 \end{equation*}
since we have already proved this statement for $k=1$ above. 
\end{proof}
\smallskip
In Section \ref{section mainproof} we will approximate equation \eqref{eq1} 
by mollifying the data 
$\psi$ and $f$, the coefficients of $\cL$ and the functions  $\eta$ and $\xi$ 
in the variable $x\in\bR^d$. It is easy to see that the mollifications of the data 
and the coefficients of $\cL$ by a nonnegative $C_0^{\infty}$ kernel of unit integral 
satisfy Assumptions \ref{assumption free} and \ref{assumption L}. 
It is less clear, however, 
that mollifications of $\eta$ and $\xi$ satisfy Assumptions \ref{assumption M} 
and \ref{assumption N}. 
We clarify this by the help of some lemmas below. 
In the rest of the paper for $\varepsilon>0$ and 
locally integrable functions $v$ defined on $\bR^d$ we use the notation 
$v^{(\varepsilon)}$ for the mollification of $v$, defined by 
\begin{equation}                                                            \label{smoothing}
v^{(\varepsilon)}(x)
=S_{\varepsilon}v(x)
:=\varepsilon^{-d}\int_{\bR^d}v(y)k((x-y)/\varepsilon)\,dy,\quad x\in\bR^d,  
\end{equation}
where $k=k(x)$ is a fixed nonnegative smooth function on $\bR^d$ 
such that $k(x)=0$ for $|x|\geq1$, $k(-x)=k(x)$ 
for $x\in\bR^d$, and $\int_{\bR^d}k(x)\,dx=1$.

\begin{lemma}                                                                      \label{lemma epsilon}
Let $\rho$ be 
a $C^k(\bR^d)$-diffeomorphism for $k\geq2$, such that 
\eqref{diffeo korlat} holds.   
Then there is a positive constant $\varepsilon_0=\varepsilon_0(M,N,d,k)$ 
such that $\rho^{(\varepsilon)}$ is a $C^{\infty}(\bR^d)$-diffeomorphism 
for every $\varepsilon\in(0,\varepsilon_0)$,  
and \eqref{diffeo korlat} remains valid for $\rho^{(\varepsilon)}$ 
in place of $\rho$, with 
$M{''}=M/2$ in place of $M$. 
\end{lemma}

\begin{proof}
We show first that $|\det D\rho^{(\varepsilon)}|$ is separated 
away from zero for sufficiently small 
$\varepsilon>0$. To this end observe that if $v=(v^1,v^2,...,v^d)$ is a Lipschitz  
function on $\bR^d$ with 
Lipschitz constant $L$, 
and in magnitude it is bounded by a constant $K$, 
then for every $\varepsilon>0$ 
$$
|\Pi_{i=1}^dv^i-\Pi_{i=1}^dv^{i(\varepsilon)}|
\leq \sum_{i=1}^dK^{d-1}|v^i-v^{i(\varepsilon)}|
\leq K^{d-1}L\varepsilon. 
$$
By virtue of this observation, taking into account that $D_i\rho^l$ is bounded by $N$ 
and it is Lipschitz 
continuous with a Lipschitz constant $N$, we get 
$$
|\det D\rho-\det D\rho^{(\varepsilon)}|\leq d!\,N^{d}\varepsilon. 
$$
Thus setting $\varepsilon'=M/(2d!\,N^{d})$,  for $\varepsilon\in(0,\varepsilon')$ we have 
$$
|\det(D\rho^{(\varepsilon)})|=|\det((D\rho)^{\varepsilon})|
\geq 
|\det(D\rho)|-|\det(D\rho)-\det(D\rho)^{(\varepsilon)}|
$$
$$
\geq |\det(D\rho)|/2\geq M/2.
$$
Clearly, $\rho^{(\varepsilon)}$ is a $C^{\infty}$ function. 
Hence by the implicit function theorem $\rho^{(\varepsilon)}$ 
is a local $C^{\infty}$-diffeomorphism for $\varepsilon\in(0,\varepsilon')$.  
We prove now that $\rho$ is a global $C^{\infty}$-diffeomorphism 
for sufficiently small $\varepsilon$. Since by the previous lemma $|D\rho^{-1}|\leq N'$,  
we have 
\begin{align*}
|x-y|\leq &N'|\rho(x)-\rho(y)|                             \\
\leq &N'|\rho^{(\varepsilon)}(x)-\rho^{(\varepsilon)}(y)|
+N'|\rho(x)-\rho^{(\varepsilon)}(x)+\rho^{(\varepsilon)}(y)-\rho(y)|
\end{align*}
for all $x,y\in\bR^d$ and $\varepsilon>0$. Observe that
\begin{align*}
|\rho(x)-\rho^{(\varepsilon)}(x)+\rho^{(\varepsilon)}(y)-\rho(y)|
&\leq\int_{\R^d}|\rho(x)-\rho(x-\varepsilon u)+\rho(y-\varepsilon u)-\rho(y)|k(u)\,du\\
&\leq 
\int_{\R^d}\int_0^1
\varepsilon|u||\nabla\rho(x-\theta\varepsilon u)-\nabla\rho(y-\theta\varepsilon u)|k(u)
\,d\theta\,du\\
&\leq\varepsilon N|x-y|\int_{|u|\leq1}|u|k(u)\,du
\leq\varepsilon N|x-y|.
\end{align*}
Thus 
$|x-y|
\leq N'|\rho^{(\varepsilon)}(x)-\rho^{(\varepsilon)}(y)|
+\varepsilon N'N|x-y|$.  
Therefore  setting $\varepsilon''=1/(2NN')$, for all $\varepsilon\in(0,\varepsilon'')$  we have 
\begin{equation}
|x-y|\leq 2N'|\rho^{(\varepsilon)}(x)-\rho^{(\varepsilon)}(y)|\quad\text{for all $x,y\in\bR^d$}, 
\end{equation}
which implies $\lim_{|x|\to\infty}|\rho^{(\varepsilon)}(x)|=\infty$, i.e., 
that  under 
$\rho^{(\varepsilon)}$ the pre-image of any compact set is a compact set 
for $\varepsilon\in(0,\varepsilon'')$. 
A continuous function with this 
property is called a {\it proper function}, and by Theorem 1 in \cite{G} a local $C^1$- 
diffeomorphism from $\bR^d$ into $\bR^d$ is a global diffeomorphism 
if and only if it is a proper function. 
Thus we have 
proved that $\rho^{(\varepsilon)}$ is a global $C^{\infty}$-diffeomorphism for each 
$\varepsilon\in(0,\varepsilon_0)$, where $\varepsilon_0=\min(\varepsilon',\varepsilon'')$. 

Now we can complete the proof of the lemma by noting that since  
$D_{j}\rho^{(\varepsilon)}=(D_{j}\rho)^{(\varepsilon)}$,   
the condition $|D^i\rho|\leq N$ implies  
$|D^i\rho^{(\varepsilon)}|\leq N$ for any $\varepsilon>0$.  
\end{proof}

Recall the definition $\tau_{\theta\eta}$ by \eqref{tau0}. 
Similarly, for each $t\in[0,T]$, $\theta\in[0,1]$ and $z\in Z$ we use the notation 
$\tau_{\theta\xi}$  for the $\bR^d$ valued function 
on $\bR^d$, defined by 
\begin{equation}                                                          \label{tau1}
\tau_{\theta\xi_{t,z}}(x)=x+\theta\xi_{t,z}(x), 
\end{equation}
for $x\in\bR^d$. To ease notation we will often omit the variables $t$ and $z$  
of $\eta$ and $\xi$.

We can apply the above lemmas to $\tau_{\theta\eta}$ and  
$\tau_{\theta\xi}$  by virtue of the following proposition. 
\begin{proposition}                                                \label{proposition tau}
Let Assumptions \ref{assumption M} and Assumptions \ref{assumption N} 
hold. 
Then for each $t\in[0,T]$, $\theta\in[0,1]$ and $z\in Z$ the functions 
$\tau_{\theta\eta}$ and 
$\tau_{\theta\xi}$ 
are $C^k(\bR^d)$-diffeomorphisms with 
$m_{\eta}$ and  $m_{\xi}$ in place of $k$, respectively. 
\end{proposition}

\begin{proof}
By the inverse function theorem 
$\tau_{\theta\eta}$ and  
$\tau_{\theta\xi}$ 
are local $C^1(\bR^d)$-diffeomor\-phisms for each $t$, $\theta$ and $z$.  
Since
$$
|\eta_{t,z}(x)|\leq \bar\eta(z)<\infty,  \quad 
|\xi_{t,z}(x)|\leq \bar\xi(z)<\infty, 
$$
we have 
$$
\lim_{|x|\to\infty}|\tau_{\theta\eta}(x)|
=\lim_{|x|\to\infty}|\tau_{\theta\xi}(x)|=\infty
$$ 
Hence $\tau_{\theta\eta}$ and  
$\tau_{\theta\xi}$ 
are global $C^1$-diffeomorphisms by Theorem 1 in \cite{G} 
for each $t\in[0,T]$, $z\in Z$ and $\theta\in[0,1]$. 
Note that by the formula on the derivative of inverse functions  
a $C^{1}(\bR^d)$-diffeomorphism and its inverse have continuous derivatives 
up to the same order. 
This observation finishes the proof of the proposition. 
 \end{proof}

\begin{corollary}                                                               \label{corollary epsilon}
Let  Assumptions \ref{assumption M} and \ref{assumption N} 
hold.  
Then Lemmas \ref{lemma chainrule} through \ref{lemma epsilon} hold 
for $\tau_{\theta\eta}$ and  
$\tau_{\theta\xi}$ 
 in place of $\rho$ 
and with $m_{\eta}$ and $m_{\xi}$ in place of $k$, respectively. 
In particular, 
there are positive constants $M=M(K,d,m)$, $N=(K,d,m)$ and $\varepsilon_0$ 
such that 
$$
M\leq  
\min(|{\rm{det}}D\tau_{\theta\eta}^{-1}|, 
|{\rm{det}}D\tau_{\theta\xi}^{-1}|)
$$
$$
M\leq  
\min(|{\rm{det}}D\tau_{\theta\eta}^{(\varepsilon)}|, 
|{\rm{det}}D\tau_{\theta\xi}^{(\varepsilon)}|, 
|{\rm{det}}D(\tau_{\theta\eta}^{(\varepsilon)})^{-1}|, 
|{\rm{det}}D(\tau_{\theta\xi}^{(\varepsilon)})^{-1}|), 
$$
$$
|D^k \tau_{\theta\eta}^{(\varepsilon)}|\leq N,\quad |D^k(\tau_{\theta\eta}^{(\varepsilon)})^{-1}|\leq N, 
\quad
|D^l \tau_{\theta\xi}^{(\varepsilon)}|\leq N,
\quad
|D^l(\tau_{\theta\xi}^{(\varepsilon)})^{-1}|\leq N  
$$
for all $\varepsilon\in(0,\varepsilon_0)$, $\theta\in[0,1]$, $(t,x,z)\in H_T\times Z$, and for 
$k=1,2,...,m_{\eta}$ and  $l=1,2,...,m_{\xi}$. 
\end{corollary}
 
\begin{lemma}                                                                                                 \label{lemma Minkowski}
Let $(S,\mathcal{S},\nu)$ be a measure space with a $\sigma$-finite measure $\nu$, 
and let $g=g(s,x)$ be a    
$\overline{\cS\otimes\cB(\R^d)}$-measurable real function  on $S\times\R^d$, 
where $\overline{\cS\otimes\cB(\R^d)}$ is the $\nu\otimes dx$-completion of the 
product $\sigma$-algebra $\cS\otimes\cB(\R^d)$.  Assume that 
$$
\int_{|x|\leq R}\int_{S}|g(s,x)|\,\nu(ds)\,dx<\infty\quad \text{for every $R>0$.}
$$
Then the following statements hold. 

\begin{enumerate}
\item[(i)]  If for a multi-number $\alpha$ the derivative 
$D_\alpha g$ of $g$ in $x$ is 
a $\overline{\cS\otimes\cB(\R^d)}$-measurable function such that 
$$
\int_S\int_{\{|x|\leq R\}}|D_\alpha g(s,x)|\,dx\,\nu(ds)<\infty 
$$
for every $R>0$, then $dx$-almost everywhere
\begin{equation}                                                                                     \label{Dg}
D_{\alpha}\int_{S}g(s,x)\,\nu(ds)=\int_{S}D_{\alpha}g(s,x)\,\nu(ds). 
\end{equation}
\item[(ii)] If $D_\alpha g$ is a $\overline{\cS\otimes\cB(\R^d)}$-measurable function 
for every multi-number $\alpha$, $|\alpha|\leq m$, such that 
$$
\int_S|g(s)|_{W^m_p}\nu (ds)<\infty,
$$
then 
\begin{equation}                                                                         \label{Minkowski} 
\left|\int_S g(s,x)\nu(ds) \right|_{W^m_p}\le \int_S |g(s)|_{W^m_p}\nu(ds).
\end{equation}
\end{enumerate}
\end{lemma}

\begin{proof} 
 Set $G(x)=\int_S g(s,x)\nu(ds)$. To prove (i) notice that 
  by the definition of generalised 
   derivatives and by Fubini's theorem 
  \begin{align*}
  \int_{\R^d}D_\alpha G(x)\varphi(x)\,dx
  &=(-1)^{|\alpha|}\int_{\R^d}\int_Sg(s,x)\,\nu(ds)D_\alpha \varphi(x)\,dx
  =\int_S\int_{\R^d}D_\alpha g(s,x) \varphi(x)\,dx\,\nu(ds)\\
  &=\int_{\R^d}\int_SD_\alpha g(s,x)\, \nu(ds)\varphi(x)\,dx
  \end{align*}
   for every $\varphi\in C^\infty_0(\R^d)$, which implies \eqref{Dg}. 
Hence 
 by H\"older's inequality 
  $$
  \left|\int_{\R^d}D_\alpha G(x)\varphi(x)\,dx \right|\le |\varphi|_{L_q} 
  \int_S |D_\alpha g(s)|_{L_p}\,\nu(ds)
  $$ 
  for every $\varphi\in C^\infty_0(\R^d)$, which implies 
  $$
  |D_\alpha G|_{L_p}
  \leq \int_S |D_\alpha g(s)|_{L_p}\nu(ds), 
  $$
  and \eqref{Minkowski} follows. 
\end{proof}    

For each $t\in[0,T]$ and $z\in Z$ define the operators $T_{t,z}$, $I=I_{t,z}$ and 
$J=J_{t,z}$ by 
\begin{equation}                                                                  \label{def T}
T_{t,z}\varphi(x)=\varphi(x+\xi_{t,z}(x)),
\end{equation}
\begin{equation}                                                                   \label{def IJ}
I_{t,z}\varphi(x)=\varphi(x+\xi_{t,z}(x))-\varphi(x), \quad
J_{t,z}\varphi(x)=\varphi(x+\eta_{t,z}(x))-\varphi(x)-\eta_{t,z}(x)\nabla\varphi(x)
\end{equation}
for $\varphi\in C^{\infty}_0(\bR^d)$. 
By Taylor's formula we have 
\begin{equation}                                                                        \label{taylor1}
I_{t,z}\varphi(x)=\int_0^1\varphi_{i}(x+\theta \xi_{t,z}(x))\xi^i_{t,z}(x)\, d\theta, 
\end{equation}
\begin{equation}                                                                        \label{taylor2}
J_{t,z}\varphi(x)
=\int_0^1
(1-\theta)\varphi_{ij}(x+\theta \eta_{t,z}(x))\eta^i_{t,z}(x)\eta^j_{t,z}(x)\,d\theta 
\end{equation}
for $\varphi\in C_0^{\infty}(\bR^d)$, 
where 
$\varphi_i=D_i\varphi$ and $\varphi_{ij}=D_{ij}\varphi$. 

\begin{lemma}                                                          \label{lemma IJ}
Let Assumptions \ref{assumption M} and \ref{assumption N} hold. Then 
$T_{t,z}\varphi(x)$, $I_{t,z}\varphi(x)$ 
and 
$J_{t,z}\varphi(x)$ are $\cB(H_T)\otimes\cZ$-measurable functions 
of $(t,x,z)\in H_T\times Z$ for each $\varphi\in C_0^{\infty}(\bR^d)$. For every 
multi-number $\alpha$  
of length $k\leq m$ we have 
\begin{equation}                                                                   \label{TI}
|D_{\alpha}T_{t,z}\varphi|_{L_p}
\leq N|\varphi|_{W^k_p}, 
\quad 
|D_{\alpha}I_{t,z}\varphi|_{L_p}\leq N\bar\xi(z)|\varphi|_{W^{k+1}_p},
\end{equation}
\begin{equation}                                                                    \label{J}
|D_{\alpha}J_{t,z}\varphi|_{L_p}\leq N\bar\eta^2(z)|\varphi|_{W^{k+2}_p}
\end{equation}
for $t\in[0,T]$, $z\in Z$ and $p\in[1,\infty)$, where $N$ is a constant depending only 
on $d,K,m,p$.  
\end{lemma}

\begin{proof}
Clearly, $T_{t,z}\varphi(x)$, $I_{t,z}\varphi(x)$ and $J_{t,z}\varphi(x)$ 
are $\cB(H_T)\otimes\cZ$-measurable functions 
by Fubini's theorem, and one can easily get estimates \eqref{TI}-\eqref{J} 
by using Lemmas \ref{lemma chainrule}  
and \ref{lemma chainrule2}, together with Lemma \ref{lemma Minkowski}. 
\end{proof}
\begin{corollary}
Let Assumptions \ref{assumption M} and \ref{assumption N} hold. Then for every 
$t,z$ the operators $T_{t,z}$, $I_{t,z}$ and $J_{t,z}$ extend to bounded linear operators 
from $W^k_p$ to $W^k_p$, from $W^{k+1}_p$ to $W^k_p$ and from $W^{k+2}_p$ to $W^k_p$,
respectively,  for $k=0,1,2,...,m$, such that 
$T_{t,z}\varphi $, 
$I_{t,z}f$ and $J_{t,z}g$ are $\cB([0,T])\otimes\cZ$-measurable $W^k_p$-valued functions of $(t,z)$ 
and 
$$
|T_{t,z}\varphi|_{W^k_p}
\leq N|\varphi|_{W^k_p}, 
\quad 
|I_{t,z}f|_{W^k_p}\leq N\bar\xi(z)|f|_{W^{k+1}_p},
\quad 
|J_{t,z}g|_{W^k_p}\leq N\bar\eta^2(z)|g|_{W^{k+2}_p}
$$
for all $\varphi\in W^k_p$, $f\in W^{k+1}_p$ and $g\in W^{k+2}_p$. \end{corollary}

\begin{proposition}                                                                     \label{proposition LMN}
Under Assumptions \ref{assumption L}, \ref{assumption M} and 
\ref{assumption N} for every integer $k\in[1,m]$ we have 
\begin{equation}                                                                      \label{LMN}
|\cL(t)v|_{W^{k-2}_p}\leq N|v|_{{W^{k}_p}},\quad |\cM(t)v|_{W^{k-2}_p}\leq N|v|_{{W^{k}_p}}, 
\quad |\cN(t)v|_{W^{k-1}_p}\leq N|v|_{{W^{k}_p}}
\end{equation}
\begin{equation}                                                                         \label{Jl}
|\cJ^{l}(t)v|_{W^{k-1}_p}\leq N|v|_{{W^{k}_p}} 
\quad l=0,1,2,...,d
\end{equation}
for all $v\in W^k_p$ and $t\in[0,T]$, where $\cJ^{l}$ for $l=0,1,...,d$ 
are defined by \eqref{def Jk}-\eqref{def J0} and $N$ is a constant, 
depending only on $d$, $m$, $p$, $K$, $T$, $K_{\eta}$ and 
$K_{\xi}$.
\end{proposition}

\begin{proof}
It is sufficient to prove the proposition for $v\in C_0^{\infty}(\bR^d)$. 
Then clearly, the statement on $\cL$ with a constant $N=N(d,K,T,m,p)$ is obvious. 		
By Taylor's formula 
$$
\cM v(x)=\int_{Z}(v(x+\eta)-v(x)-\eta\nabla v(x))\,\mu(dz)
=\int_Z\int_0^1(1-\theta)v_{ij}(x+\theta \eta)\eta^i\eta^j d\theta\, \mu(dz).
$$
Hence, due to Assumption \ref{assumption M}, by Lemma \ref{lemma chainrule2} 
for $k\in[2,m]$ we get
\begin{equation*}
|\cM v|_{W^{k-2}_p}
\leq \int_Z\int_0^1\left| v_{ij}(\cdot+\theta \eta)\right|_{W^{k-2}_p}\bar{\eta}^2
\,d\theta\, \mu(dz)\leq | v|_{W^{k}_p}\int_Z\bar{\eta}^2(z)\,\mu(dz),  
\end{equation*}
which proves \eqref{LMN} for $\cM$ when $k\geq2$. For every $\varphi$ by integration by parts we have 
$$
(\cM v,\varphi)=I_1+I_2+I_3 
$$
with 
$$
I_1:=\int_Z\int_0^1(\theta-1)\int_{\bR^d}v_{j}(x+\theta \eta)\eta^i_i\eta^j\varphi(x)\,dx d\theta\, \mu(dz)
\leq N|v|_{W^1_p}|\varphi|_{L_q}\int_Z\bar\eta^2(z)\,dz
$$
$$
I_2:=\int_Z\int_0^1(\theta-1)\int_{\bR^d}v_{j}(x+\theta \eta)\eta^i\eta^j_i\varphi(x)\,dx d\theta\, \mu(dz)
\leq N|v|_{W^1_p}|\varphi|_{L_q}\int_Z\bar\eta^2(z)\,dz
$$
$$
I_3:=\int_Z\int_0^1(\theta-1)\int_{\bR^d}v_{j}(x+\theta \eta)\eta^i\eta^j\varphi_i(x)\,dx d\theta\, \mu(dz)
\leq N|v|_{W^1_p}|\varphi|_{W^1_q}\int_Z\bar\eta^2(z)\,dz. 
$$
Hence there is a positive constant $N=N(K,p,d, K_{\eta})$ such that 
$$
(\cM v,\varphi)\leq N|v|_{W^1_p}|\varphi|_{W^1_q}
$$
for any $v\in W^1_p$ and $\varphi\in C_0^{\infty}(\bR^d)$, which proves \eqref{LMN} for $\cM$ when 
$k=1$. 
For $\mathcal{N}$ we have 
	$$
	\mathcal{N}v(x)=\int_{Z}(v(x+\xi_{t,z}(x)-v(x))\,\nu(dz)
	=\int_Z\int_0^1 \xi_{t,z}(x)\nabla v(x+\theta \xi_{t,z}(x))\, d\theta\,\nu(dz).
	$$
Proceeding as before, using Assumption \ref{assumption N} 
we get \eqref{LMN} for $\cN$. Estimates \eqref{Jl} can be proved similarly. 
\end{proof}

\begin{lemma}                                                                  \label{lemma intIJ}
Let Assumptions \ref{assumption M} and \ref{assumption N} hold with $m=0$. Then 
\begin{equation}                                                             \label{intI}
\int_{\bR^d}I_{t,z}\varphi(x)\,dx\leq N\bar\xi(z)\,|\varphi|_{L_1}, 
\end{equation} 
\begin{equation}                                                             \label{intJ}
\int_{\bR^d}J_{t,z}\phi(x)\,dx\leq N\bar\eta^2(z)|\phi|_{L_1}   
\end{equation}   
for $\varphi\in W^1_1$ and $\phi\in W^2_1$ with a constant $N=N(K,d)$. 
\end{lemma}
 
\begin{proof}
The proof of \eqref{intJ} is given in \cite{Da}. 
For the convenience of the reader we prove both estimates here. 
We may assume that $\varphi, \phi\in C_0^{\infty}$.  
For each $(t,z,\theta)\in[0,T]\times Z\times [0,1]$ 
let $\pi^{-1}_{t,z,\theta}$ and $\tau^{-1}_{t,z,\theta}$ denote 
the inverse of the functions $x\to x+\theta\xi_{t,z}(x)$ 
and $x\to x+\theta\eta_{t,z}(x)$, respectively. 
Using \eqref{taylor1} and \eqref{taylor2} by change of variables we have 
\begin{align}
\int_{\bR^d}I_{t,z}\varphi(x)\,dx
=&\int_0^1\int_{\bR^d}\nabla\varphi(x)\chi_{t,z,\theta}(x)\,dx\,d\theta                                          \label{iI}\\
\int_{\bR^d}J_{t,z}\phi(x)\,dx
=&\int_0^1\int_{\bR^d}(1-\theta)D_{ij}\phi(x)\varrho^{ij}_{t,z,\theta}(x)\,dx\,d\theta                     \label{iJ}
\end{align}
with
\begin{align*}
\chi_{t,\theta,z}(x)=&\xi_{t,z}(\pi^{-1}_{t,z,\theta}(x))|
{\rm{det}}D\pi^{-1}_{t,z,\theta}(x)|, \\
\varrho^{ij}_{t,z,\theta}(x)
=&\eta^i_{t,z}(\tau^{-1}_{t,z,\theta}(x))
\eta^j_{t,z}(\tau_{t,z,\theta}(x)) |{\rm{det}}D\tau^{-1}_{t,z,\theta}(x)|. 
\end{align*}
Due to Assumptions \ref{assumption M} and \ref{assumption N}, 
using Corollary \ref{corollary epsilon} we have a constant $N=N(K,d)$ 
such that 
$$
|D\chi_{t,\theta,z}(x)|\leq N\bar\xi(z),\quad |D_{ij}\varrho^{ij}_{t,z,\theta}(x)|\leq N\bar\eta^2(z)
$$
for all $(t,z,\theta)\in[0,T]\times Z\times[0,1]$. Thus from \eqref{iI} and \eqref{iJ} by integration by parts 
we get \eqref{intI} and \eqref{intJ}. 
\end{proof}

Next we present a special case of Theorem 2.1 
from \cite{K} on the $L_p$-norm of semimartingales 
with values in Sobolev spaces, where we use the notation 
$D^{\ast}_{\alpha}=-D_k$ for $\alpha=k=1,2,...,d$, 
and $D^{\ast}_0=D_0$ stands for the identity operator.  

\begin{lemma}                                                                          \label{lemma Kr}
Let $\psi\in L_p(\bR^d)$, $u\in L_p([0,T],W^1_p(\bR^d))$ 
and  
$f^{\alpha}\in L_p([0,T],L_p(\bR^d))$ for some $p\geq2$, 
for $\alpha=0,1,...,d$,  
such that for each $\varphi\in C_0^{\infty}(\bR^d)$ 
$$
\int_{\bR^d}u(t)\varphi\,dx=\int_{\bR^d}\psi\varphi\,dx
+\int_0^t\int_{\bR^d} f^{\alpha}(s)
D_{\alpha}^{\ast}\varphi\,dx\,ds\quad\text{$dt$-almost everywhere}.
$$
Then, there is an $L_p(\bR^d)$-valued  continuous function 
$\tilde u$ such that $\tilde u(t)=u(t)$ for $dt$-a.e., and 
\begin{equation}                                                                                           \label{formula1}
|\tilde u_t|^p_{L_p}=|\psi|_{L_p}^p
+\int_0^t\int_{\bR^d}p|u(s)|^{p-2}u(s)f^0(s)-p(p-1)|u(s)|^{p-2}D_iu(s)f^i(s)\,dx\,ds
\end{equation}
for all $t\in[0,T]$, where the repeated index $i$ means summation over $i=1,2,...,d$. 
\end{lemma}

The following lemma is a vector-valued version of a special case of 
Lemma 5.1 from \cite{K}. Its proof is a simple exercise left for the reader. 
\begin{lemma}                                                                          \label{lemma Kr multi}
Let $\psi^{\alpha}\in L_p$, 
$u^{\alpha}\in L_p([0,T],W^1_p(\bR^d))$ 
and  
$f^{\alpha}\in L_p([0,T],L_p(\bR^d))$ 
for some $p\geq2$,
for $\alpha\in A$ for a finite index set $A$, such that  
for each $\varphi\in C_0^{\infty}(\bR^d)$ and $\alpha\in A$ 
$$
\int_{\bR^d}u^{\alpha}(t)\varphi\,dx=\int_{\bR^d}\psi^{\alpha}\varphi\,dx
+\int_0^t\int_{\bR^d} f^{\alpha}(s)\varphi\,dx\,ds \quad\text{$dt$-almost everywhere}.
$$
Then for every $\alpha\in A$ there is an $L_p(\bR^d)$-valued  continuous function 
$\tilde u^{\alpha}$ on $[0,T]$, 
such that $\tilde u^{\alpha}(t)=u^{\alpha}(t)$ for $dt$-almost every $t\in[0,T]$, and 
\begin{equation}                                                               \label{formula2}
\|\tilde u_t\|^p_{L_p}=\|\psi\|_{L_p}^p
+\sum_{\alpha\in A}\int_0^t\int_{\bR^d}p|u(s)|^{p-2}u^{\alpha}(s)f^{\alpha}(s)\,dx\,ds
\end{equation}
holds for all $t\in[0,T]$, where 
$|\tilde u|:=\big(\sum_{\alpha}(\tilde u^\alpha)^{2}\big)^{1/2}$ 
and   $|u|:=\big(\sum_{\alpha}(u^\alpha)^{2}\big)^{1/2}$. 
\end{lemma}

\mysection{ $L^p$ estimates}                                           \label{section estimates}

In this section we prove estimate \eqref{estimatesup} 
for $p=2^k$ and $m=n$ for integers $k\geq1$, 
and $n\geq0$, provided  Assumptions \ref{assumption L}, \ref{assumption M} 
and \ref{assumption N} hold with sufficiently high $m$.  
To this end we use the formula
$$
\|D^nu(t)\|^p_{L_p}=\|D^n\psi\|^p_{L_p}
$$
\begin{equation}                                                                           \label{Lp}
+p\int_0^t
\sum_{|\alpha|=n}
(|D^nu|^{p-2}D_{\alpha}u(s),D_{\alpha}\cA u(s)
+D_{\alpha}f(s))\,ds
\end{equation}
for $W^{n+2}_p$-valued solutions $u$ to equation \eqref{eq1},   
which we obtain by an application of Lemma \ref{lemma Kr multi} 
with $\psi^{\alpha}$, $D_{\alpha}u$ and $D_{\alpha}\cA u+D_{\alpha}f$ in place of 
$\psi^{\alpha}$, $u^{\alpha}$ and $f^{\alpha}$, respectively. 

To estimate the right-hand side of \eqref{Lp}, the crucial result is the following. 
\begin{theorem}                                                                      \label{theorem crucial}
Let Assumptions \ref{assumption L}, \ref{assumption M} 
and \ref{assumption N} hold. Then for $p=2^k$, $n=0,1,...,m$
$$
\sum_{|\alpha|=n}(|D^nv|^{p-2}D_{\alpha}v,D_{\alpha}\cA v)\leq N|v|^p_{W^n_p}
$$
for all $v\in W^{n+2}_p$ with a constant $N=N(d,p,m,K,K_{\xi},K_{\eta})$.  
\end{theorem}
We prove this theorem after some lemmas.
\begin{lemma}                                                                  \label{lemma L}
Let Assumption \ref{assumption L} hold. Then for $p=2^k$ for $n=0,1,...,m$
\begin{equation}                                                               \label{estimate L}
Q(v):=\sum_{|\alpha|=n}(|D^nv|^{p-2}D_{\alpha}v,D_{\alpha}\cL v)\leq N|v|^p_{W^n_p}
\end{equation}
for all $v\in W^{n+2}_p$ with a constant $N=N(d,p,m,K)$.  
\end{lemma}
\begin{proof}
This lemma can be obtained from general estimates
given in \cite{GGK}. Here we give a direct proof of it. 
For functions $g$ and $h$ on $\bR^d$ we write 
$
g\sim h
$
if they have identical integrals over $\mathbb R^d$,  and 
we write $g\preceq h$ if $g\sim h+\tilde h$ such that 
the integral of $\tilde h$ over $\bR^d$ can be estimated by the right-hand side of 
\eqref{estimate L}. 
Consider first the case $n=0$. It is easy to see that 
\begin{align*}
p|v|^{p-2}
v \cL v&\preceq p|v|^{p-2}v(a^{ij}v_{ij}+b^iv_i)                                   \nonumber\\
&\sim- p(p-1)|v|^{p-2}a^{ij}v_{i}v_{j}-a^{ij}_j(|v|^{p})_i+(|v|^p)_ib^i      \nonumber \\
&\sim- p(p-1)|v|^{p-2}a^{ij}v_{i}v_{j}+(a^{ij}_{ji}-b^i_i)|v|^p                  \nonumber\\
&\preceq  -p(p-1)|v|^{p-2}a^{ij}v_{i}v_{j},                                                                    \label{1.18.11.13}
\end{align*}
where, and later on, we use the notation $g_{\alpha}:=D_{\alpha}g$ for functions $g$ over 
$\bR^d$ and multi-numbers $\alpha=\alpha_1\dots\alpha_n$. 
This by virtue  of Assumption \ref{assumption L} proves \eqref{estimate L} when $n=0$. 
Let us now estimate $Q$ when $n\geq1$. Then it is easy to see that 
$$
A:=p|D^nv|^{p-2}\sum_{|\alpha|=n}
v_{\alpha}D_{\alpha}\cL v
$$
\begin{equation}                                                                 \label{1.21.13}
\preceq p|D^nv|^{p-2}\sum_{|\alpha|=n}(
v_{\alpha}a^{ij}v_{\alpha ij}+
\sum_{l=1}^nv_{\alpha}a^{ij}_{\alpha(l)}D_{ij}v_{\bar\alpha(l)}
+v_{\alpha}b^iv_{\alpha i}), 
\end{equation}
where $\alpha(l)$ denotes the $l$-th element of multi-number $\alpha$, and 
$\bar\alpha(l)$ is the multi-number we get from $\alpha$ by leaving out its $l$-th 
element. 
Notice that 
\begin{equation*}                                           
2v_{\alpha}a^{ij}v_{ij\alpha}
=a^{ij}[|D^nv|^2]_{ij}-2a^{ij}v_{i\alpha}v_{j\alpha}, 
\quad2v_{\alpha}b^{i}v_{i\alpha}=b^{i}(|D^nv|^2)_i. 
\end{equation*}
Hence integrating by parts and using Assumption \ref{assumption L}, 
with $c_p=p(p-2)/4\geq0$ we have 
$$
p|D^nv|^{p-2}\sum_{|\alpha|=n}
v_{\alpha}a^{ij}v_{\alpha ij}
=                                 
\tfrac{p}{2}|D^nv|^{p-2}(a^{ij}[|D^nv|^2]_{ij}
-2a^{ij}v_{i\alpha}v_{j\alpha} )
$$
$$                   
\sim 
-c_p|Dv|^{p-4}a^{ij}[|D^nv|^2]_i[|D^nv|^2]_j   
-\tfrac{p}{2}a^{ij}_j|D^nv|^{p-2}[|D^nv|^2]_i                      
-p|D^nv|^{p-2}a^{ij}v_{i\alpha}v_{j\alpha}                  
$$
\begin{equation}                                                                       \label{4.14.6.13} 
\preceq a^{ij}_{ji}|D^nv|^{p}-p|D^nv|^{p-2}a^{ij}v_{i\alpha}v_{j\alpha}                                         
\preceq                          
-p|D^nv|^{p-2}a^{ij}v_{i\alpha}v_{j\alpha},                              
\end{equation}
and 
\begin{equation}                                                                         \label{1.23.11.13}
p|D^nv|^{p-2}\sum_{|\alpha|=n}
v_{\alpha}b^{i}v_{i\alpha}= \tfrac{p}{2}|D^nv|^{p-2}b^i
(|D^nv|^2)_i = b^i(|D^nv|^{p})_i \sim-b^i_i|D^nv|^p \preceq0.                               
\end{equation} 
Taking into account \eqref{4.14.6.13} and 
\eqref{1.23.11.13}, 
from \eqref{1.21.13} we get   
\begin{equation}                                                                        \label{A}
A\preceq -p|D^nv|^{p-2}
a^{ij}v_{i\alpha}v_{j\alpha}+B                                             
\end{equation}
with
$$
B:=p|D^nv|^{p-2}\sum_{|\alpha|=n}
v_{\alpha}\sum_{l=1}^{n}
a^{ij}_{\alpha(l)}D_{ij}v_{\bar\alpha(l)}.
$$
We estimate $B$ by using the simple inequality
$$
|v_{\alpha}\sum_{l=1}^{n}
a^{ij}_{\alpha(l)}D_{ij}v_{\bar\alpha(l)}|
\leq \varepsilon^{-1}|v_{\alpha}|^2
+\varepsilon n\sum_{l=1}^{n}|a^{ij}_{\alpha(l)}D_{ij}v_{\bar\alpha(l)}|^2
$$
 for every $\varepsilon>0$ and multi-number $\alpha$, to get 
\begin{equation}                                                     \label{epsilon}
B\leq p\varepsilon^{-1} |D^nv|^{p}
+\varepsilon {n} p|D^nv|^{p-2}C
\quad
\text{with}
\quad 
C:=\sum_{|\alpha|=n}\sum_{l=1}^{n}
|a^{ij}_{\alpha(l)}D_{ij}v_{\bar\alpha(l)}|^2. 
\end{equation}
It is well-known, see e.g. \cite{OR2},  that for symmetric matrices 
$V\in\bR^{d\times d}$ and functions 
$a=(a^{ij})$ mapping $\bR^d$ into the space of 
 symmetric non-negative definite $d\times d$ matrices, such that the second order derivatives 
of $a$ are bounded by a constant $L$, 
the inequality 
\begin{equation*}                                                     
|D_la^{ij}V^{ij}|^2\leq N a^{ij}V^{ik}V^{jk}
\end{equation*}
holds for any $l\in\{1,2,...,d\}$, where $N$ is a constant depending only on $L$ and $d$.   
Using this   with $V^{ij}:=D_{ij}v_{\bar\alpha(l)}$ 
for each $l=1,2,...,n$ and multi-number $\alpha$ of length $n$, we get  
$$
C\leq N\sum_{|\alpha|=n}\sum_{l=1}^{n}
a^{ij}D_{ik}v_{\bar\alpha(l)}D_{jk}v_{\bar\alpha(l)}
\leq N'\sum_{|\alpha|=n}
a^{ij}D_{i}v_{\alpha}D_{j}v_{\alpha}
$$
with a constant $N'=N'(d,K,n)$. Thus, choosing $\varepsilon$ sufficiently small in 
\eqref{epsilon}, from \eqref{A} we obtain $A\preceq 0$, which proves the lemma. 
\end{proof}

For the following lemmas recall the definition of the operators 
$I=I_{t,z}$ and $J=J_{t,z}$ by \eqref{def IJ}, and notice 
that the 
 identities 
\begin{equation}                                                \label{identity I}
2vIv=Iv^2-(Iv)^2
\end{equation}
\begin{equation}                                                \label{identity J}
2vJv=Jv^2-(I^{\eta}v)^2
\end{equation} 
hold for $(t,x,z)\in H_T\times Z$, 
for functions $v=v(x)$ of $x\in\bR^d$, where 
\begin{equation*}                                                                \label{S}
I^{\eta}v(x)=v(x+\eta(x))-v(x). 
\end{equation*}
\begin{lemma}                                                                  \label{lemma N}
Let Assumption \ref{assumption N} hold. 
Then for $p=2^k$ for $n=0,1,...,m$
\begin{equation}                                                               \label{estimate I}
{\mathcal I}(v)
:=\sum_{|\alpha|=n}(|D^nv|^{p-2}D_{\alpha}v,D_{\alpha}Iv)
\leq N{\bar\xi}|v|^p_{W^n_p}
\end{equation}
for all $v\in W^{n+1}_p$ with a constant $N=N(d,p,m,K)$.  
\end{lemma}
\begin{proof}
Consider first the case $n=0$. Then by identity \eqref{identity I} 
$$                                                                    
|v|^{p-2}vIv= \tfrac{1}{2}|v|^{p-2}Iv^2-\tfrac{1}{2}|v|^{p-2}(Iv)^2=
\tfrac{1}{2}|v|^{p-4}v^2Iv^2-\tfrac{1}{2}|v|^{p-2}(Iv)^2
$$
\begin{equation}                                                                       \label{recursion I}  
=\tfrac{1}{4}|v|^{p-4}Iv^4-\tfrac{1}{2}|v|^{p-2}(Iv)^2-\tfrac{1}{4}|v|^{p-4}(Iv^2)^2
=\dots=\tfrac{1}{p}Iv^p-A
\end{equation}
with 
$$
A=\sum_{j=1}^k2^{-j}|v|^{p-2^j}(Iv^{j})^2\geq0. 
$$
Hence integrating over $\bR^d$, by \eqref{intI} we have 
$$
{\mathcal I}(v)\leq \frac{1}{p}\int_{\bR^d}Iv^p\,dx\leq N{\bar\xi}|v|^p_{L_p}. 
$$
Assume now that $n\geq 1$ and let $\alpha$ be a multi-number of length $n$. 
Let $T$ denote the operator defined by $Tg(x)=g(x+\xi(x))$ on functions $g=g(x)$  
of $x\in\bR^d$. Then 
$$
(Iv)_k=Iv_k+\xi^i_kTv_i,\quad (Tv)_k=Tv_k+\xi^i_kTv_i
$$
for $k=1,2,...,d$. (Recall that we use the notation $g_{\alpha}=D_{\alpha}g$ for multi-numbers 
$\alpha$.) Hence, by induction on the length $n$ of the multi-number of $\alpha$, we obtain 
$$
(Iv)_{\alpha}=Iv_{\alpha}+\sum_{1\leq |\beta|\leq n}q^{\alpha,\beta}Tv_{\beta}, 
$$
with some polynomial $q^{\alpha,\beta}$ of 
$\{\xi^i_{\gamma}:1\leq|\gamma|\leq n,\,i=1,...,d\}$ for each multi-number $\beta$ of length between 
1 and $n$. 
The degree of these polynomials is not greater than $n$, their constant term is zero, and the other 
coefficients are nonnegative integers. Hence 
$$
|D^nv|^{p-2}v_{\alpha}(Iv)_{\alpha}=|D^nv|^{p-2}v_{\alpha}Iv_{\alpha}
+\sum_{1\leq |\beta|\leq n}|D^nv|^{p-2}v_{\alpha}q^{\alpha,\beta}Tv_{\beta}, 
$$
where the repeated multi-numbers $\alpha$ mean summation over $|\alpha|=n$.  
By using the same calculation as in \eqref{recursion I} we have 
$$
|D^nv|^{p-2}v_{\alpha}Iv_{\alpha}
= \tfrac{1}{2}|D^nv|^{p-2}\{I(|D^nv|^2)-\sum_{|\alpha|=n}(Iv_{\alpha})^2\}
$$
$$
\leq \tfrac{1}{2}|D^nv|^{p-4}|D^nv|^2I(|D^nv|^2)\leq\dots\leq\tfrac{1}{p}I(|D^nv|^p).  
$$
Thus 
$$
|D^nv|^{p-2}v_{\alpha}(Iv)_{\alpha}
\leq \tfrac{1}{p}I(|D^nv|^p)+N{\bar\xi}|D^nv|^{p-1}\sum_{1\leq |\beta|\leq n}|Tv_{\beta}|
$$
\begin{equation}
\leq \tfrac{1}{p}I(|D^nv|^p)+N{\bar\xi}|D^nv|^{p}
+N'{\bar\xi} \sum_{1\leq |\beta|\leq n}|Tv_{\beta}|^p
\end{equation}
with constants $N$ and $N'$ depending only on $m,d$ and $p$. 
Integrating here over $\bR^d$  we get \eqref{estimate I}. 
\end{proof}
\begin{lemma}                                                                  \label{lemma J}
Let Assumption \ref{assumption M} hold. Then for $p=2^k$ 
for integers $k\geq1$, and  for $n=0,1,...,m$ we have 
\begin{equation}                                                               \label{estimate J}
\frJ(v):=\sum_{|\alpha|=n}(|D^nv|^{p-2}D_{\alpha}v,D_{\alpha}Jv)
\leq N{\bar\eta}^2|v|^p_{W^n_p}
\end{equation}
for all $v\in W^{n+2}_p$ and $(t,z)\in[0,T]\times\bR^d$ with a constant $N=N(d,p,m,K)$.  
\end{lemma}
\begin{proof}
Consider first the case $n=0$. Then using identity \eqref{identity J} and proceeding with the proof 
in the same way as in the proof of the previous lemma we get 
\begin{equation}                                                                      \label{recursion J}
v^{p-2}vJv=\tfrac{1}{2}v^{p-2}J(v^2)-\tfrac{1}{2}v^{p-2}(I^\eta(v))^2
=\dots=\tfrac{1}{p}Jv^p-B
\end{equation}
with 
$$
B=\sum_{j=1}^k2^{-j}|v|^{p-2^j}(I^{\eta}v^{j})^2\geq0. 
$$
Integrating here over $\bR^d$ by \eqref{intJ} we have 
$$
\frJ(v)\leq \tfrac{1}{p}\int_{\bR^d}J(v^p)\,dx\leq N{\bar\eta}^2|v|^p_{L_p}. 
$$
Assume now that $n\geq 1$ and let $\alpha$ be a multi-number of length $n$. 
Let $Tg$ denote the operator defined by 
$Tg(x):=g(x+\eta(x))$. Then 
for $(Tv)_k:=D_k(Tv)$, $(I^{\eta}v)_k:=D_k(I^{\eta}v)$ and $(Jv)_k:=D_k(Jv)$ we have 
\begin{equation*}
(Tv)_k=Tv_k+\eta^i_kTv_i,\quad (Iv)_k=Iv_k+\eta^i_kTv_i, 
\quad (Jv)_k=Jv_k+\eta^i_kIv_i
\end{equation*}
for every $k=1,...,d$, where, and later on within the proof we write $I$ in place of $I^{\eta}$ to ease notation. 
 Hence by induction on the length of $\alpha$ we get 
$$
(Jv)_{\alpha}=Jv_{\alpha}
+\sum_{1\leq |\beta|\leq n}p^{\alpha,\beta}Iv_{\beta}
+\sum_{1\leq |\beta|\leq n}q^{\alpha,\beta}Tv_{\beta}, 
$$
with some polynomials $p^{\alpha,\beta}$ and $q^{\alpha,\beta}$ of 
$\{\eta^i_{\gamma}: 1\leq|\gamma|\leq n, i=1,...,d\}$. 
The degree of these polynomials is not greater than $n$, their constant term is zero, 
the coefficients of each first order term in the polynomials 
$q^{\alpha,\beta}$ is also zero, all the 
other coefficients in $p^{\alpha,\beta}$ and $q^{\alpha,\beta}$ 
are nonnegative integers. Hence we get 
\begin{equation}                                                                             \label{Ja}
|D^nv|^{p-2}v_{\alpha}(Jv)_{\alpha}=|D^nv|^{p-2}v_{\alpha}Jv_{\alpha}
+A^{\beta}+B^{\beta}
\end{equation}
with 
$$
A^{\beta}:=|D^nv|^{p-2}v_{\alpha}p^{\alpha,\beta}Iv_{\beta}, 
\quad
B^{\beta}:= |D^nv|^{p-2}v_{\alpha}q^{\alpha,\beta}Tv_{\beta},  
$$
where repeated $\alpha$ means summation 
over the multi-numbers $\alpha$ of length $n$.  

Clearly,  for all $\beta$ we have 
$$
|B^{\beta}|\leq N{\bar\eta}^2|Dv|^{p-1}|Tv_{\beta}|
$$
with constants $N=N(m,K,d).$ 
For $|\beta|\leq n-1$ we estimate $A^{\beta}$ in the same way to get 
$$
|A^{\beta}|\leq N\bar{\eta}|D^nv|^{p-1}|Iv_{\beta}|,   
$$
and for $|\beta|=n$  we use Young's inequality 
to write 
$$
|v_{\alpha}p^{\alpha,\beta}Iv_{\beta}|
\leq \varepsilon|Iv_{\beta}|^2
+\varepsilon^{-1}|v_{\alpha}p^{\alpha,\beta}|^2
\leq \varepsilon|Iv_{\beta}|^2
+\varepsilon^{-1}|D^nv|^2\sum_{|\alpha|=n}|p^{\alpha,\beta}|^2
$$
$$
\leq \varepsilon|Iv_{\beta}|^2
+{N}{\varepsilon}^{-1}{\bar\eta}^2|D^nv|^2. 
$$
Hence for $|\beta|=n$ we have 
$$
|A^{\beta}|\leq \varepsilon |D^nv|^{p-2}|Iv_{\beta}|^2 
+{N}{\varepsilon}^{-1}{\bar\eta}^2|D^nv|^{p}
$$
for $\varepsilon>0$ with a constant $N=N(K,d,m)$. 
Calculating as in \eqref{recursion J} we obtain 
$$
|D^nv|^{p-2}v_{\alpha}Jv_{\alpha}
\leq\tfrac{1}{2}|D^nv|^{p-2}\{J(|D^nv|^2)-\sum_{|\alpha|=n}|Iv_{\alpha}|^2\}
$$
$$
 \leq\tfrac{1}{4}|D^nv|^{p-4}|D^nv|^2J(|D^nv|^2)
 -\tfrac{1}{2}\sum_{|\alpha|=n}|D^nv|^{p-2}|Iv_{\alpha}|^2
 \leq\dots
$$
$$
\leq\tfrac{1}{p}J(|D^nv|^p)
-\tfrac{1}{2}\sum_{|\alpha|=n}|D^nv|^{p-2}|Iv_{\alpha}|^2.
$$
Using these estimates, from \eqref{Ja} we obtain 
$$
|D^nv|^{p-2}v_{\alpha}(Jv)_{\alpha}
\leq \tfrac{1}{p}J(|D^nv|^p)-\tfrac{1}{2}\sum_{|\alpha|=n}|D^nv|^{p-2}|Iv_{\alpha}|^2
+\varepsilon |D^nv|^{p-2}\sum_{|\beta|=n}|Iv_{\beta}|^2 
$$
$$
+N{\varepsilon}^{-1}{\bar\eta}^2|D^nv|^{p}+
N{\bar\eta}^2|D^nv|^{p-1}\sum_{1\leq |\beta|\leq n}|Tv_{\beta}|
+N{\bar\eta}|D^nv|^{p-1}\sum_{1\leq |\beta|\leq n-1}|Iv_{\beta}|. 
$$
Choosing here 
$\varepsilon=1/2$, we get 
\begin{align}
|D^nv|^{p-2}v_{\alpha}(Jv)_{\alpha}
\leq &\tfrac{1}{p}J(|D^nv|^p)
+N{\bar\eta}^2|Dv|^{p}                                                                           \nonumber\\
&+N{\bar\eta}^2|D^nv|^{p-1}\sum_{1\leq |\beta|\leq n}|Tv_{\beta}|
+N{\bar\eta}|D^nv|^{p-1}
\sum_{1\leq |\gamma|\leq n-1}|Iv_{\gamma}|.                                         \label{inequality J}
\end{align}
By H\"older's inequality, 
taking into account \eqref{TI} we have 
\begin{align} 
\int_{\bR^d}|D^nv|^{p-1}|Tv_{\beta}|\,dx
&\leq N|v|^{p-1}_{W^n_p}|Tv_{\beta}|_{L_p}
\leq N'|v|^{p}_{W^n_p} \quad\text{for $|\beta|\leq n$},                  \nonumber\\
\int_{\bR^d}|D^nv|^{p-1}|Iv_{\gamma}| \,dx
&\leq 
N|v|^{p-1}_{W^n_p}|Iv_{\gamma}|_{L_p}
\leq N'\bar\eta |v|^{p}_{W^n_p}\quad\text{for $|\gamma|\leq n-1$}    \label{intTS}
\end{align}
with some constants $N=N(d,p)$ and $N'=N'(d,m,p,K)$. 
Integrating inequality \eqref{inequality J} over $\bR^d$ 
and using inequalities \eqref{intJ} and \eqref{intTS} we obtain \eqref{estimate J}. 
\end{proof}
\begin{proof}[Proof of Theorem \ref{theorem crucial}]
By the definition of $\cA$ for $v\in W^{n+2}_p$ we have 
$$
\sum_{|\alpha|=n}(|D^nv|^{p-2}D_{\alpha}v,D_{\alpha}\cA v)
=\sum_{|\alpha|=n}(|D^nv|^{p-2}D_{\alpha}v,D_{\alpha}\cL v)
$$
$$
+\sum_{|\alpha|=n}(|D^nv|^{p-2}D_{\alpha}v,D_{\alpha}\cR v)
+\int_Z\mathcal I(v)\nu(dz)+\int_Z{\frJ}(v)\mu(dz),   
$$
where $\mathcal I$ and ${\frJ}$ are defined 
in \eqref{estimate I} and \eqref{estimate J}, respectively. 
Due to Assumption \ref{assumption R} by the Cauchy-Schwarz 
and H\"older inequalities we obtain 
$$
\sum_{|\alpha|=n}(|D^nv|^{p-2}D_{\alpha}v,D_{\alpha}\cR v)
\leq (|D^nv|^{p-1},|D^n\cR v|)\leq |D^nv|_{L_p}^{p-1}|D^n\cR v|_{L_p}
\leq N|v|^p_{W^n_p}
$$
with a constant $N=N(d,p,m,K)$. 
Hence we get Theorem \ref{theorem crucial} 
by Lemmas \ref{lemma L}, \ref{lemma N} and \ref{lemma J}. 
\end{proof}

\mysection{Proof of the main result}                                            \label{section mainproof}

\subsection{Uniqueness of the generalised solution when $p=2^k$.}
Define
\begin{align*}
&Q^{(1)}(v):=\int_{\bR^d}|v|^{p-2}v(\bar b^iD_iv+cv)-(p-1)|v|^{p-2}D_iva^{ij}D_jv\,dx\\
&Q^{(2)}(v):=(|v|^{p-2}v,\cN v)\\
&Q^{(3)}(v):=\int_{\bR^d}|v|^{p-2}v\cJ^0v-(p-1)|v|^{p-2}D_kv\cJ^kv\,dx, 
\end{align*}
Notice that under 
Assumptions \ref{assumption L},  
\ref{assumption M} 
and \ref{assumption N} with $m=1$ and $p=2^k$ by Lemmas \ref{lemma L}, 
\ref{lemma N} and \ref{lemma J}
we have 
\begin{equation}                                           \label{N1N2}
Q^{(1)}(v)=(|v|^{p-2}v, \cL v)\leq N_1|v|^p_{L_p}, 
\quad Q^{(2)}(v)\leq N_2|v|^p_{L_p}, 
\end{equation}
\begin{equation}                                             \label{N3}
Q^{(3)}(v)=(|v|^{p-2}v, \cM v)\leq N_3|v|^p_{L_p}
\end{equation}
for $v\in C_0^{\infty}$ with constants $N_1=N_1(d,p,K)$, $N_2(d,p,K,K_{\xi})$ 
and $N_3(d,p,K,K_{\eta})$.  
We show below that these estimates 
hold also for $v\in W^1_p$. 
To this end we use the following lemma 
from \cite{K}. 

\begin{lemma}                                                                                               \label{lemma Krylov}
Let $(S,\mathcal{S}, \nu)$ be a measure space, and let $\{v_n\}_{n\in\N}$ be 
a sequence of real-valued $\cS$-measurable functions defined on $S$ such that 
such that $v_n\rightarrow v$ in the measure $\nu$, and 
	$$
	\int|v_n|^rd\nu\rightarrow \int|v|^r d\nu.
	$$
for some $r>0$. Then $\int|v_n-v|^rd\nu \rightarrow 0$ as $n\to \infty$.	
\end{lemma}
\begin{proposition}                                                 \label{proposition Qform}
Let Assumptions \ref{assumption L}, \ref{assumption M} 
and \ref{assumption N} 
hold with $m=0$.  Then for integers $k\geq1$ and $p=2^k$ 
\begin{equation}                                                                      \label{Qform}
Q^{(i)}(v)\leq N_i|v|^p_{L_p} \quad i=1,2,3
\end{equation}
for all $v\in W^1_p$ and $p=2^k$ for integers $k\geq1$ 
with the same constants $N_i$ as in \eqref{N1N2}- \eqref{N3}.  
\end{proposition}
\begin{proof}
Let $\{v_n\}_{n=1}^{\infty}$ be a sequence of $C_0^{\infty}$ functions, 
which converges in the $W^1_p$ norm 
to some $v\in W^1_p$ as $n\to \infty$. 
We claim that $Q^{(3)}(v_n)\to Q^{(3)}(v)$. If $p=2$ then clearly 
$$
Q^{(3)}(v_n)=(v_n,\cJ^0v_n)-(D_kv_n,\cJ^kv_n)
\to
(v,\cJ^0v)-(D_kv,\cJ^kv)=Q^{(3)}(v), 
$$
since $D_k$, $\cJ^k$ and $\cJ_0$ are bounded linear operators from $W^1_2$ into $L_2$, 
and the inner product $(\varphi,\phi)$ in $L_2$ is continuous in $\varphi,\phi\in L_2$. 
Assume now that $p=2^k$ for $k\geq1$. 
By choosing subsequences we may assume that $v_n\to v$ also almost surely. Clearly, 
$$
Q^{(3)}(v_n)-Q^{(3)}(v)= (p-1)B_n+C_n
$$ 
with 
$$
B_n:=\int_{\R^d}(|v|^{p-2}D_kv\cJ^kv-|v_n|^{p-2}D_kv_n\mathcal{J}^kv_n)\, dx
$$
$$
C_n:=\int_{\R^d}(|v_n|^{p-2}v\mathcal{J}^0v_n-|v_n|^{p-2}v\mathcal{J}^0v)\, dx. 
$$
Observe that $B_n=B_n^{(1)}+B_n^{(2)}+B_n^{(3)}$ with 
\begin{align*}
B_n^{(1)}:=&\int_{\R^d}(|v|^{p-2}-|v_n|^{p-2})D_kv\mathcal{J}^kv\,dx \\
B_n^{(2)}:=&\int_{\R^d}|v_n|^{p-2}(D_kv-D_kv_n)\mathcal{J}^kv\,dx  \\
B_n^{(3)}:=&\int_{\R^d}|v_n|^{p-2}D_kv_n(\mathcal{J}^kv-\mathcal{J}^kv_n)\, dx.
\end{align*}
On the one hand, by H\"older's inequality,
\begin{align*}
|B_n^{(1)}|\leq &||v_n|^{p-2}-|v|^{p-2}|_{L^\frac{p}{p-2}}|D_kv|_{L^p}|\mathcal{J}^kv|_{L^p},\\
|B_n^{(2)}|\leq&|v_n|_{L^p}^{p-2}|D_kv-D_kv_n|_{L^p}|\mathcal{J}^kv|_{L^p}  \\
|B_n^{(3)}|\leq&|v_n|^{p-2}_{L_p}|D_kv_n|_{L_p}|\mathcal{J}^kv-\mathcal{J}^kv_n|_{L_p}. 
\end{align*}
Since $v_n\to v$ in $W^1_p$, it is easy to see that $B^{(i)}_n\to0$ for $i=2,3$. 
By Lemma \ref{lemma Krylov} we have 
$$
||v_n|^{p-2}-|v|^{p-2}|_{L^\frac{p}{p-2}}\to0,
$$
which gives $\lim_{n\to\infty}B^{(1)}_n=0$. 
We get in the same way that $\lim_{n\to\infty}C_n=0$. 
Hence we finish the proof of the lemma for $i=3$ by letting $n\to \infty$ 
in the inequality \eqref{Qform} 
with $v_n$ in place of $v$. The estimates for $i=1,2$ can be proved similarly. 
\end{proof}

 Now we can prove the uniqueness of a generalised solution in the case $p=2^k$. 
 Let $u$ and $v$ be $W^1_p$-valued solutions of equation \eqref{eq1}. 
 Then $w=u-v$ solves \eqref{eq1} with $w(0)=0$ and $f=0$. 
 Hence, by the $L^p$ formula \eqref{formula1} we have 
\begin{align*}
|w(t)|_{L^p}^p=p\int_0^t\sum_{k=1}^4Q^{(k)}(w(s))\,ds, 
\end{align*}
where $Q^{(4)}$ is defined by 
$$
Q^{(4)}(v)=(|v|^{p-2}v,\cR(v))
$$
for $v\in W^1_p$. Due to Assumption \ref{assumption R} we have 
$|Q^{(4)}(v)|\leq K|v|^p_{L_p}$. Hence,  by taking into account Proposition 
\ref{proposition Qform} 
we get a constant $N$ such that 
$$
|w(t)|_{L^p}^p\leq N\int_0^t|w(s)|^p_{L_p}\,ds,  
$$
which proves $u=v$.

\subsection{Existence of a generalised solution}

In the whole subsection we assume that 
Assumptions \ref{assumption L} through \ref{assumption free} 
hold with $m\geq1$ and $p\geq2$. 
We prove the existence of a solution to equation \eqref{eq1} with initial condition $u(0)=\psi$ 
in several steps below. In the first three steps we assume that $p=2^k$ for some integer $k\geq1$ 
and that $m$ is an integer. We construct a solution $u$ in $L_p([0,T], W^m_p)$ 
by approximation procedures, 
and estimate its norm in $L_p([0,T], W^s_p)$ for integers $s=0,1,...,m$ (for $p=2^k$)  
by the right-hand side of \eqref{estimatesup}. Hence, using standard results from 
interpolation theory we prove the existence of a generalised solution 
$u\in L_p([0,T],V^m_p)$ when $p\geq2$ 
and $m\geq1$ are any real numbers. Moreover, we show that 
$u\in C([0,T], V^s_p)\cap C_w([0,T], V^m_p)$ for every $s<m$, 
and obtain also the estimate \eqref{estimatesup}. We note that similar interpolation arguments 
are used in \cite{GG} to obtain estimates in $L_p$-spaces for solutions of stochastic finite difference 
schemes.

{\it Step 1.} First, in addition to Assumptions \ref{assumption L}, \ref{assumption M}, \ref{assumption N} 
and \ref{assumption free}, we assume that 
$\psi$ and $f$ are compactly supported, and that $\mu(Z)<\infty$ and $\nu(Z)<\infty$. Under 
these assumptions we approximate the Cauchy problem \eqref{eq1} 
with initial condition $u(0)=\psi$ by smoothing the data and the coefficients in the problem. 
Recall that for $\varepsilon>0$ and functions $v$ on $\bR^d$ the notation $v^{(\varepsilon)}$ 
means the mollification $v^{(\varepsilon)}=S_{\varepsilon}v$ of $v$ defined in 
\eqref{smoothing}. 
We consider the Cauchy problem 
\begin{align}                                                                                              
dv(t,x)=&(\mathcal{A}_{\varepsilon}v(t,x) 
+f^{(\varepsilon)}(t,x))\,dt, 
\quad (t,x)\in H_T ,                   \label{eq2}\\
v(0,x)=&\psi^{(\varepsilon)}(x),\quad x\in\bR^d                                                               \label{ini2}
\end{align} 
for $\varepsilon\in(0,\varepsilon_0)$, where $\varepsilon_0$ is given in 
Corollary \ref{corollary epsilon}, and  
$$
\mathcal{A}_{\varepsilon}:=\mathcal{L}_{\varepsilon}
+\cM_{\varepsilon}+\cN_{\varepsilon}+\cR_{\varepsilon}
$$
with operator $\mathcal{R}_{\varepsilon}={S}_{\varepsilon}\mathcal{R}$ 
and operators $\mathcal{L}_{\varepsilon}$, $\cM_{\varepsilon}$ 
and $\cN_{\varepsilon}$, defined by 
$$
\mathcal{L}_{\varepsilon}
=a_{\varepsilon}^{ij}D_{ij}+b^{(\varepsilon)i}D_i+c^{(\varepsilon)}, 
\quad 
a_{\varepsilon}=a^{(\varepsilon)}+\varepsilon \bI,
$$
$$
\cM_{\varepsilon}\varphi(x)
=\int_{Z}\{\varphi(x+\eta^{(\varepsilon)}_{t,z})-\varphi(x)
-\eta^{(\varepsilon)}_{t,z}\nabla\varphi(x)\}\,\mu(dz),
$$
$$
\cN_{\varepsilon}\varphi=\int_{Z}\{\varphi(x+\xi^{(\varepsilon)}_{t,z})-\varphi(x)\}\,\nu(dz)
$$
for $\varphi\in C_0^{\infty}(\bR^d)$. (Recall that $\bI$ denotes the $d\times d$ unit matrix.)  	
	
Since $\psi^{(\varepsilon)}$ and $f^{(\varepsilon)}$ 
are compactly supported, they belong to $W^n_2$ for every $n\ge 0$. 
By standard results of the $L_2$-theory of parabolic PDEs, \eqref{eq2}-\eqref{ini2} has  
a unique solution $u_{\varepsilon}$, which is a continuous  $W^n_2$-valued function of $t\in[0,T]$ 
for every $n\ge 0$ (see, e.g., \cite{KR1977} or \cite{Ro}).  
Thus for any $\varphi\in C_0^{\infty}$ we have 
$$
(u_{\varepsilon}(t),\varphi)=
(\psi^{(\varepsilon)},\varphi)
$$
$$
+\int_0^t-(a_{\varepsilon}^{ij}D_ju_{\varepsilon}(s), D_i\varphi)
+({\bar b}^{i(\varepsilon)}D_iu_{\varepsilon}(s)
+c^{(\varepsilon)}u_{\varepsilon}(s)+\cR_{\varepsilon}u_{\varepsilon}(s)+f^{(\varepsilon)}(s),\varphi)\,ds 
$$
\begin{equation}                                                              \label{eqe}
+\int_0^t-(\cJ_{\varepsilon}^{i}u_{\varepsilon}(s),D_i\varphi)
+(\cJ_{\varepsilon}^{0}u_{\varepsilon}(s),\varphi)+(\cN_{\varepsilon} u_{\varepsilon}(s),\varphi)\,ds
\end{equation}
for $t\in[0,T]$, where  $\cJ_{\varepsilon}^{i}$ and  $\cJ_{\varepsilon}^{0}$ are defined 
as $\cJ^{i}$ and $\cJ^{0}$, respectively in \eqref{def J0}, but with $\eta^{k(\varepsilon)}$ 
and $\eta^{l(\varepsilon)}_k$ in place of $\eta^{k}$ and $\eta^{l}_k$, respectively, 
for $k,l=1,2,...,d$. Notice that \eqref{eqe} can be rewritten as 
$$ 
(u_{\varepsilon}(t),\varphi)=
(\psi^{(\varepsilon)},\varphi)
+\int_0^t
(\cA_{\varepsilon}u_{\varepsilon}(s)
+f^{(\varepsilon)}(s),\varphi)\,ds, 
\quad
t\in[0,T],\quad\varphi\in C_0^{\infty}, 
$$
and, equivalently,  as 
$$
(D_{\alpha}u_{\varepsilon}(t),\varphi)=
(D_{\alpha}\psi^{(\varepsilon)},\varphi)
+\int_0^t(D_{\alpha}\cA_{\varepsilon}u_{\varepsilon}(s)
+D_{\alpha}f^{(\varepsilon)}(s),\varphi)\,ds  
\quad
t\in[0,T],\quad\varphi\in C_0^{\infty}
$$
for all multi-numbers $\alpha$ of length $n$.   
By Sobolev embedding $u_{\varepsilon}$ is a continuous  $W^n_p$-valued 
function for every $n\ge 0$ and 
$p\geq2$. Hence by Lemma \ref{lemma Kr multi} we have 
$$
||D^nu_{\varepsilon}||^p_{L_p}
=||D^n\psi^{(\varepsilon)}||^p_{L_p}
$$
$$
+p\int_0^t\sum_{|\alpha|=n}
(|D^{n}u_{\varepsilon}(s)|^{p-2}D_{\alpha}u_{\varepsilon}(s),
D_{\alpha}\cA_{\varepsilon}u_{\varepsilon}(s)
+D_{\alpha}f^{(\varepsilon)}(s))\,ds  
$$
for $p=2^k$, which by Theorem \ref{theorem crucial}, 
known properties of mollifications and Young's inequality gives 
$$
||D^nu_{\varepsilon}||^p_{L_p}
\leq 
||D^n\psi||^p_{L_p}
+N\int_0^t|u_{\varepsilon}(s)|^p_{W^n_p}+
(|D^{n}u_{\varepsilon}(s)|^{p-2}D_{\alpha}u_{\varepsilon}(s),D_{\alpha}f^{(\varepsilon)}(s))\,ds  
$$
$$
\leq ||D^n\psi||^p_{L_p}
+N\int_0^t\{|u_{\varepsilon}(s)|^p_{W^n_p}+\tfrac{p-1}{p}|u_{\varepsilon}(s)|^p_{W^n_p} 
+\tfrac{1}{p}|f(s)|^p_{W^n_p}\}\,ds. 
$$
This via Gronwall's lemma implies that for $\varepsilon\in(0,\varepsilon_{0})$ 
\begin{equation}                                                                      \label{gronwall}
\sup_{t\in[0,T]}|u_{\varepsilon}(t)|^p_{W^n_{p}}
\leq 
N\left(|\psi|^p_{W^n_{p}}+\int_0^T|f(t)|^p_{W^{n}_{p}}dt\right)
\end{equation}
for every $n\geq0$ and $p=2^k$, for integers $k\geq1$, 
with a constant $N=N(T,p,d,n,K,K_{\xi},K_{\eta})$.
For 
$r>1$ and $p\geq2$ we denote by $\mathbb{W}^n_{p,r}$ the  space of $W^n_p$-valued functions 
$v$ of $t\in[0,T]$ such that 
	$$
	|v|_{\mathbb{W}^n_{p,r}}:=\left( \int_0^T |v(t)|_{W^n_p}^{r}dt\right)^{1/r}<\infty.
	$$
We use also the notation $\mathbb W^n_p$ and $\mathbb L_p$ for 
$\mathbb W^n_{p,p}$ and $\mathbb W^0_{p,p}$, respectively. 
Observe that with this norm $\mathbb{W}^n_{p,r}$ is a reflexive Banach space,  
and from \eqref{gronwall} we have 
\begin{equation}                                                                       \label{becsles}
|u_{\varepsilon}|^p_{{\mathbb W}^n_{p,r}}
\leq 
N\left(|\psi|^p_{W^n_{p}}+\int_0^T|f(t)|^p_{W^{n}_{p}}dt\right)
\end{equation}
for all $\varepsilon\in(0,\varepsilon_0)$, $p=2^k$, $r>1$ and $n=0,1,2,...,m$, 
with a constant 
$N$ depending only on $T$, $p$, $d$, $m$, $K$, $K_{\xi}$ and $K_{\eta}$.  
Hence there exists a sequence of positive numbers $\{\varepsilon_k\}_{k\in\N}$ 
such that $\varepsilon_k\to 0$ for $k\to\infty$, 
and $u_{\varepsilon_k}$ converges weakly to a function $u$ in $\mathbb{W}^n_{p,r}$ 
for every $n=0,1,\dots,m$ and integers $r>1$. From \eqref{becsles} we get 
$$
|u|^p_{\mathbb W^n_{p,r}}\leq \liminf_{k\to\infty}|u_{\varepsilon_k}|^p_{\mathbb W^n_{p,r}}
\leq N\left(|\psi|^p_{W^n_{p}}+\int_0^T|f(t)|^p_{W^{n}_{p}}dt\right).   
$$

Our aim now is to pass to the limit in equation \eqref{eqe} along 
$\varepsilon_k\to0$. To this end we 
take a real-valued bounded Borel function $h$ of $t\in[0,T]$, 
multiply 
both sides of equation \eqref{eqe} with $h(t)$ 
and then integrate it against $dt$ over $[0,T]$. Thus 
for a fixed $\varphi\in C_0^{\infty}$ and taking 
$\varepsilon_k$ in place of $\varepsilon$, we obtain
\begin{equation}                                                               \label{eqek}
F(u_{\varepsilon_k})=
\int_0^T(\psi^{(\varepsilon_k)},\varphi)h(t)\,dt
+\sum_{i=1}^5F^i_k(u_{\varepsilon_k})
+\int_0^T\int_0^t(f^{(\varepsilon_k)}(s),\varphi)h(t)\,ds\,dt, 
\end{equation}
where $F$ and $F^i_k$ 
are functionals defined for $v\in\mathbb W^{1}_p$ 
by
\begin{equation}                              \label{F}
F(v)=\int_0^T(v(t),\varphi)h(t)\,dt,
\end{equation}
$$
F^1_k(v)=\int_0^Th(t)\int_0^t-(a_{\varepsilon_k}^{ij}D_jv(s), D_i\varphi)
+({\bar b}^{i(\varepsilon_k)}D_iv(s)
+c^{(\varepsilon_k)}v(s),\varphi)\,ds\,dt,
$$
$$
F^2_k(v)=-\int_0^Th(t)\int_0^t(\cJ_{\varepsilon_k}^{i}v(s),D_i\varphi)\,ds\,dt,
$$
$$
F^3_k(v)=\int_0^Th(t)\int_0^t(\cJ_{\varepsilon_k}^{0}v(s),\varphi)\,ds\,dt, 
$$
$$
F^4_k(v)=\int_0^Th(t)\int_0^t(\cN_{\varepsilon_k} v(s),\varphi)\,ds\,dt,  
$$
$$
F^5_k(v)=\int_0^Th(t)\int_0^t(\cR_{\varepsilon_k} v(s),\varphi)\,ds\,dt. 
$$
For each $i$ define also the functional $F^i$ in the same way as $F^i_k$ is defined above, but with 
$a$, $b$, $c$, $\cJ^i$, $\cJ^0$, $\cN$ and $\cR$ 
in place of $a_{\varepsilon_k}$, $b^{(\varepsilon_k)}$, $c^{(\varepsilon_k)}$, 
$\cJ^i_{\varepsilon_k}$, $\cJ^0_{\varepsilon_k}$ and $\cN_{\varepsilon_k}$, 
$\cR_{\varepsilon_k}$, respectively.  
Clearly, due to the boundedness of $h$ we have a constant $C$ such that 
for all $v\in \mathbb W^1_p$
$$
F(v)\leq C|v|_{\mathbb L_p}|\varphi|_{L_q}
\leq C|v|_{\mathbb W^1_p}|\varphi|_{L_q}, 
$$
where $q=p/(p-1)$. This means  $F\in \mathbb W^{1\ast}_p$, 
the Banach space of bounded linear functionals on $\mathbb W^1_p$. 
To take the limit $k\to\infty$  in equation \eqref{eqek} 
we show below that  $F^i_k$ and $F^i$ are in $\mathbb W_p^{1\ast}$,  and 
$F^i_k\to F^i$ strongly in $(\mathbb W^1_p)^{\ast}$, 
for every $i$ as $k\to\infty$. 

Since the functions  $h$, $a_{\varepsilon}$, ${\bar b}^{(\varepsilon)}$ and $c^{(\varepsilon)}$ 
are in magnitude bounded by a constant, by H\"older's inequality we have  
$$
|F^1_k(v)|\leq N |v|_{\mathbb W^1_p}|\varphi|_{W^1_q}, 
$$
with a constant $N$ independent of $v$, which shows that $F^1_k\in \mathbb W^{1\ast}_p$ for all $k$. 
In the same way we get 
$F^1\in \mathbb W^{1\ast}_p$. Since $|h|$ is bounded by a constant, 
by simple estimates and using H\"older's inequality we have 
$$
|F^2_k(v)|\leq 
C\int_0^T\int_0^1\int_{Z}\int_{\bR^d}
(|v(s,x+\theta\eta_{s,z}^{(\varepsilon_k)}(x))|+|v(s,x)|)\bar\eta(z)|D\varphi(x)|\,dx\,\mu(dz)\,d\theta\,dt
$$
$$
\leq C(A+\mu^{1/p}(Z)|v|_{\mathbb W^1_p})B|\varphi|_{W^1_q}, 
$$
where 
$$
A^p=\int_0^T\int_0^1\int_{Z}\int_{\bR^d}
(|v(s,x+\theta\eta^{(\varepsilon_k)}(x))|^p\,dx\,\mu(dz)\,d\theta\,dt,   
\quad
B^q=\int_{Z}\bar\eta^q(z)\,\mu(dz)<\infty,
$$
and $C$ is a constant, independent of $v$. By the change of variable 
$$
y=\tau^{(\varepsilon_k)}_{\theta\eta}(x)=x+\theta\eta_{t,z}^{(\varepsilon_k)}(x), 
$$
and taking 
into account that by virtue of Corollary 
\ref{corollary epsilon} $|{\rm det}\,D(\tau_{\theta\eta}^{(\varepsilon)})^{-1}|$ 
is bounded by a constant, uniformly 
in $t,\theta,z$, we get 
$$
A\leq C\mu^{1/p}(Z)|v|_{\mathbb W^1_p}
$$
with a constant $C$ independent of $v$. 
Consequently, there is a constant $\bar C$ such that 
$$
|F^2_k(v)|\leq \bar C|v|_{\mathbb W^1_p}|\varphi|_{W^1_q} 
$$
for all $v\in\mathbb W^1_p$, i.e., $F^2_k\in \mathbb W^{1\ast}_p$ for every $k\geq1$. 
We can prove in the same way that $F^{2}$, $F^i_k\in \mathbb W^{1\ast}_p$ and 
$F^i\in \mathbb W^{1\ast}_p$ for $i=3,4$ and $k\geq1$.  It is easy to see 
that $F^5_k\in \mathbb W^{1\ast}_p$ and 
$F^5\in \mathbb W^{1\ast}_p$.  To prove $F^1_k\to F^1$ 
notice that since $h$ is bounded by a constant $N$,  we have 
$$
|F^1_k(v)-F^1(v)|\leq N\sum_{i=1}^3A^i_k(v)
$$
for all $k\geq1$ with 
$$
A^1_k(v):=\int_0^T\int_{\bR^d}|D_jv(s,x)||a^{ij}_{\varepsilon_k}(s,x)-a^{ij}(s,x)||D_i\varphi(x)|\,dx\,ds, 
$$
$$
A^2_k(v):=\int_0^T\int_{\bR^d}|v(s,x)||{\bar b}^{i(\varepsilon_k)}(s,x)-{\bar b}^{i}(s,x)||D_i\varphi(x)|)\,dx\,ds
$$
$$
A^3_k(v):=\int_0^T\int_{\bR^d}|v(s,x)||c^{(\varepsilon_k)}(s,x)-c(s,x)||\varphi(x)|\,dx\,ds. 
$$
By H\"older's inequality 
$$
\sup_{|v|_{\mathbb W^1_p}\leq1}A^1_k(v)\leq \big||a_{\varepsilon_k}-a||D\varphi|\big|_{\mathbb L_q},
\quad
\sup_{|v|_{\mathbb W^1_p}\leq1}A^2_k(v)\leq \big||b^{(\varepsilon_k)}-b||D\varphi|\big|_{\mathbb L_q}, 
$$
$$
\sup_{|v|_{\mathbb W^1_p}\leq1}A^3_k(v)\leq \big|(c^{(\varepsilon_k)}-c)\varphi\big|_{\mathbb L_q}, 
$$
where $\mathbb L_q=\mathbb W^0_{q,q}$. Letting here $k\to\infty$ we get 
$$
\lim_{k\to\infty}\sup_{|v|_{\mathbb W^1_p}\leq1}A^i_k(v)=0
$$
for $i=1,2,3$, by virtue of Lebesgue's theorem on dominated 
convergence, which proves that $F^1_k\to F^1$ strongly in $\mathbb W^{1\ast}_p$ as $k\to\infty$. 
Next notice that by using the boundedness of $h$ and by changing variables we have 
\begin{equation}                                                     \label{F2}
|F^2_k(v)-F^2(v)|\leq C(|B^1_k(v)|+|B^2_k(v)|)
\end{equation}
with
$$
B^1_k(v):=\int_0^T\int_0^1\int_{Z}(v(s),\gamma^k_{\theta,z}(s)-\gamma_{\theta,z}(s))\,
\mu(dz)\,d\theta\,ds,
$$
$$
B^2_k(v):=\int_0^T\int_0^1\int_{Z}(v(s),\delta^k_{z}(s)-\delta_{z}(s))\,
\mu(dz)\,d\theta\,ds, 
$$
and a constant $C$, independent of $v$ and $k$, where 
$$
\gamma^k_{\theta,z}(s,x)=\eta_{s,z}^{i(\varepsilon_k)}((\tau_{\theta\eta}^{(\varepsilon_k)})^{-1}(x))
\varphi_i
((\tau^{(\varepsilon_k)}_{\theta\eta})^{-1}(x))|{\rm det}\,D(\tau^{(\varepsilon_k)}_{\theta\eta})^{-1}(x)|, 
$$
$$
\gamma_{\theta,z}(s,x)=\eta^{i}_{s,z}(\tau^{-1}_{\theta\eta}(x))
\varphi_i(\tau^{-1}_{\theta\eta}(x))|{\rm det}\,D\tau^{-1}_{\theta\eta}(x)|, 
$$
$$
\delta^k_{z}(s,x)):=\eta^{i(\varepsilon_k)}_{s,z}(x)\varphi_i(x), 
\quad
\delta_{z}(s,x):=\eta^{i}_{s,z}(x)\varphi_i(x),   
$$
and $\varphi_i:=D_i\varphi$.
By H\"older's inequality
\begin{equation}                                               \label{B}
\sup_{|v|_{\mathbb W^1_p}\leq1}|B^1_k(v)|
\leq \mu^{\tfrac{1}{p}}(Z)
\Big(
\int_Z\int_0^1|\gamma^k_{\theta,z}-\gamma_{\theta,z}|^q_{\mathbb L_q}
\,\mu(dz)\,d\theta
\Big)^{1/q}
\end{equation}
$$
\sup_{|v|_{\mathbb W^1_p}\leq1}|B^2_k(v)|
\leq \mu^{\tfrac{1}{p}}(Z)
\Big(
\int_Z\int_0^1|\delta^k_{\theta,z}-\delta_{\theta,z}|^q_{\mathbb L_q}
\,\mu(dz)\,d\theta
\Big)^{1/q},  
$$
where $q=p/(p-1)\leq 2$. 
Observe that 
$$
|\tau(\tau^{(\varepsilon)})^{-1}(x)-x|
=|\tau(\tau^{(\varepsilon)})^{-1}(x)-\tau^{(\varepsilon)}(\tau^{(\varepsilon)})^{-1}(x))|
\leq 
\sup_{y\in\bR^d}|\tau(y)-\tau^{(\varepsilon)}(y)|
$$
\begin{equation*}                                                          \label{convergence0}
\leq\sup_{y\in\bR^d}|\eta^{(\varepsilon)}(y)-\eta(y)|
\leq \sup_{y\in\bR^d}\int_{\bR^d}|\eta(y-\varepsilon z)-\eta(y)|k(z)\,dz
\leq K\varepsilon,  
\end{equation*}
where, and later on, we write $\tau$ instead of $\tau_{\theta\eta}$ 
to ease notation. 
Hence 
$$
\lim_{\varepsilon\downarrow0}|\tau(\tau^{(\varepsilon)})^{-1}(x)-x|=0, 
$$
which implies that for every $s, \theta,z$
\begin{equation}                                                  \label{convergence1}
\lim_{\varepsilon\downarrow0}(\tau_{\theta\eta}^{(\varepsilon)})^{-1}(x)
=(\tau_{\theta\eta})^{-1}(x)
\quad\text{for every $x\in\bR^d$.}
\end{equation}
Similarly, 
$$
|D\tau^{(\varepsilon)}(\tau^{(\varepsilon)})^{-1}(x))-D\tau(\tau^{(\varepsilon)})^{-1}(x))|\leq 
\sup_{y\in\bR^d}|D\tau^{(\varepsilon)}(y)-D\tau(y)|
$$
$$
\leq \sup_{y\in\bR^d}\int_{\bR^d}|D\eta(y-\varepsilon z)-D\eta(y)|k(z)\,dz
\leq dK\varepsilon.
$$
Hence, taking into account \eqref{convergence1}, we have 
$$
\lim_{\varepsilon\downarrow0}
|D\tau^{(\varepsilon)}(\tau^{(\varepsilon)})^{-1}(x)-D\tau\tau^{-1}(x)|
$$
\begin{equation*}
\leq \limsup_{\varepsilon\downarrow0}
|D\tau^{(\varepsilon)}(\tau^{(\varepsilon)})^{-1}(x))-D\tau(\tau^{(\varepsilon)})^{-1}(x))|
+\lim_{\varepsilon\downarrow0}|D\tau(\tau^{(\varepsilon)})^{-1}(x))-D\tau(\tau^{-1}(x))|=0. 
\end{equation*}
Thus for every $s, \theta,z$ and $x$ 
$$
\lim_{\varepsilon\downarrow0}|{\rm{det}} D(\tau^{(\varepsilon)})^{-1}(x)|
=\lim_{\varepsilon\downarrow0}
|{\rm{det}} D\tau^{(\varepsilon)}(\tau^{(\varepsilon)})^{-1}(x)|^{-1}
$$
\begin{equation*}                                             \label{convergence2}
=|{\rm{det}} D\tau(\tau^{-1}(x))|^{-1}
=|{\rm{det}} D\tau^{-1}(x)|. 
\end{equation*}
Hence, using also 
 \eqref{convergence1}  we have 
$$
\lim_{k\to\infty}\gamma^k_{\theta,z}(s,x)=\gamma_{\theta,z}(s,x)
$$
for $\theta\in[0,1]$, $z\in Z$ and  $(s,x)\in H_T$. Observe also that 
\begin{equation}                                                                            \label{gammak}
|\gamma^k_{\theta,z}|^q_{\mathbb L_q}
=\int_0^T\int_{\bR^d}
|\eta^{i(\varepsilon_k)}(x)\varphi_i(x)|^q|{\rm det}\, 
D\tau^{(\varepsilon_k)}_{\theta\eta}(x)|^{1-q}\,dx\,ds, 
\end{equation}   
\begin{equation}                                                                            \label{gamma}
|\gamma_{\theta,z}|^q_{\mathbb L_q}
=\int_0^T\int_{\bR^d}
|\eta^{i}(x)\varphi_i(x)|^q|{\rm det}\, D\tau_{\theta\eta}(x)|^{1-q}\,dx\,ds.    
\end{equation}
By virtue of Corollary \ref{corollary epsilon} we can use here Lebesgue's theorem on dominated  
convergence to get  
$$
\lim_{k\to\infty}|\gamma^k_{\theta,z}|^q_{\mathbb L_q}=|\gamma_{\theta,z}|^q_{\mathbb L_q} 
$$
for each $\theta\in(0,1)$ and $z\in Z$. 
Thus by Lemma \ref{lemma Krylov} we have 
$$
\lim_{k\to\infty}|\gamma^k_{\theta,z}-\gamma_{\theta,z}|^q_{\mathbb L_q}=0
\quad 
\text{for every $(\theta,z)$}.
$$
Notice that by \eqref{gammak}-\eqref{gamma} and by virtue 
of Corollary \ref{corollary epsilon} the function  
$|\gamma^k_{\theta,z}-\gamma_{\theta,z}|^q_{\mathbb L_q}$ of $(\theta,z)$ can be estimated 
by a constant times $\bar\eta^{q}$, which has finite integral with respect to $\mu$.
Therefore letting $k\to\infty$ in \eqref{B} we obtain 
$$
\lim_{k\to\infty}B^1_k(v)=0
$$
by Lebesgue's theorem on dominated convergence. We get in the same way 
$$
\lim_{k\to\infty}B^2_k(v)=0. 
$$
Consequently, letting $k\to\infty$ in \eqref{F2} we get 
$$
\lim_{k\to\infty}\sup_{|v|_{\mathbb W^1_p}\leq 1}|F^2_k(v)-F^2(v)|=0, 
$$
which means $F^2_k\to F^2$ strongly in $\mathbb W^{1\ast}_p$. 
We can prove similarly that $F^i_k\to F^i$ strongly in $\mathbb W^{1\ast}_p$ for $i=3,4$. 
It is easy to see that this holds also for $i=5$. 
Thus due to the convergence of $u_{\varepsilon_k}$ to $u$ weakly in $\mathbb W^1_p$, we 
have 
$$
\lim_{k\to\infty}F(u_{\varepsilon_k})=F(u),
\quad 
\lim_{k\to\infty}F^{i}_k(u_{\varepsilon_k})=F(u)\quad \text{for $i=1,2,3,4,5$}.  
$$
Clearly,
$$
\lim_{k\to\infty}\int_0^T(\psi^{(\varepsilon_k)},\varphi)\,dt
=\int_0^T(\psi,\varphi)\,dt, 
$$
$$
 \lim_{k\to\infty}\int_0^T\int_0^t(f^{(\varepsilon_k)}(s),\varphi)\,ds\,dt
 =\int_0^T\int_0^t(f(s),\varphi)\,ds\,dt. 
$$
Thus taking $k\to\infty$ in equation \eqref{eqek} we obtain 
\begin{equation}                                                               \label{limit1}
F(u)=
\int_0^T(\psi,\varphi)h(t)\,dt
+\sum_{i=1}^5F^i(u)
+\int_0^T\int_0^t(f(s),\varphi)h(t)\,ds\,dt.  
\end{equation}
This means for every bounded real function $h$ the function $u:[0,T]\to W^1_p$ 
satisfies the equation
$$
	\int_0^Th(t)(u(t),\varphi)\,dt=\int_0^Th(t)(\psi,\varphi)\,dt
	+\int_0^Th(t)\int_0^t\langle\cA u(s),\varphi\rangle+(f(s),\varphi)\,ds \,dt
$$
for every $\varphi\in C_0^{\infty}(\bR^d)$. Thus for each $\varphi\in C_0^{\infty}$ 
equation \eqref{solution} holds for $dt$-almost every $t\in[0,T]$. Hence taking into account 
that $u\in L_p([0,T],W^1_p)$, by Lemma \ref{lemma Kr} $u$ has a modification, denoted 
also by $u$, which is continuous as an $L_p$-valued function and it is the solution 
of equation \eqref{eq1} with initial value $\psi$.

{\it Step 2.} We are going to dispense with the additional assumption 
that $\mu$ and $\nu$ are finite measures, i.e., we assume now 
that Assumptions \ref{assumption L} through \ref{assumption free} hold 
with $m\geq1$ and $f(t,x)$ and $\psi(x)$ vanish for $|x|\geq R$ for some $R>0$. 

Since $\mu$ and $\nu$  are $\sigma$-finite, there is a sequence 
$(Z_n)_{n=1}^{\infty}$ of sets $Z_n\in\cZ$ such that $Z_i\subset Z_{i+1}$ 
for $i\geq1$, $Z=\cup_{n=1}^{\infty}Z_n$, and $\mu(Z_n)<\infty$  
and $\nu(Z_n)<\infty$ for all $n$. For each $n$ define the measures $\mu_n$ 
and $\nu_n$  by 
$$
\mu_n(A)=\mu(A\cap Z_n),\quad \nu_n(A)=\nu(A\cap Z_n)\quad A\in\cZ,   
$$
and consider the equation 
\begin{equation}                                                                    
\frac{\partial}{\partial t} u(t,x)=(\mathcal{L}+\cM_n+\cN_n+\cR) u(t,x) +f(t,x)
\end{equation}
with initial condition $u(0)=\psi$, where $\cM_n$ and $\cN_n$ 
are defined as $\cM$ and $\cN$ 
in \eqref{def M} and in \eqref{def N}, but with $\mu_n$ and $\nu_n$ 
n place of $\mu$ and $\nu$, respectively. 
By virtue of Step 1 for each $n$ there is a solution $u_n$ 
to this problem in the sense that 
\begin{equation}                                                                   \label{phin}
F(u_n)=
\int_0^T(\psi,\varphi)h(t)\,dt+\Phi^{1}(u_n)
+\sum_{i=2}^4\Phi_n^i(u_n)+\Phi^5(u_n)
+\int_0^T\int_0^t(f(s),\varphi)h(t)\,ds\,dt  
\end{equation}
holds with arbitrary $\varphi\in C_0^{\infty}$ and $h\in L_{\infty}([0,T],\bR)$, where  
the functional $F$ is defined by \eqref{F}, as before, and 
$$
\Phi^1(v)=\int_0^Th(t)\int_0^t-(a^{ij}D_jv(s), D_i\varphi)
+({\bar b}^iD_iv(s)
+cv(s),\varphi)\,ds\,dt,
$$
$$
\Phi^2_n(v)=-\int_0^Th(t)\int_0^t(\cJ_n^{i}v(s),D_i\varphi)\,ds\,dt,
$$
$$
\Phi^3_n(v)=\int_0^Th(t)\int_0^t(\cJ_{n}^{0}v(s),\varphi)\,ds\,dt, 
$$
$$
\Phi^4_n(v)=\int_0^Th(t)\int_0^t(\cN_{n} v(s),\varphi)\,ds\,dt 
$$
$$
\Phi^5(v)=\int_0^Th(t)\int_0^t(\cR v(s),\varphi)\,ds\,dt 
$$
for $v\in \mathbb W^1_p$. Here $\cJ_{n}^{i}$ and  $\cJ_{n}^{0}$ are defined 
as $\cJ^{i}$ and $\cJ^{0}$ respectively in \eqref{def J0}, but with $\mu_n$ 
in place of $\mu$. Just as for $F^1_k$ before, 
we can see that $\Phi^{1}\in\mathbb W^{1\ast}_p$. 
Let $\Phi^j$ be defined as $\Phi^j_n$ for $j=2,3,4$
above, but with $\cJ^{i}$ and $\cN$ in place of $\cJ^{i}_n$ and $\cN_n$ 
for $i=0,1,...,d$ in their definition. 
By Step 1 $u_n\in \mathbb W^j_{p,r}$ for $j=0,1,...,m$, $p=2^k$  
for integers $k\geq1$, $r\in(1,\infty)$, 
and for all $n$ 
\begin{equation}                                                            \label{un estimate}
|u_n|^p_{\mathbb W^j_{p,r}}\leq N(|\psi|^p_{W^j_p}+|f|^p_{\mathbb W^j_{p}}), 
\quad \text{for $p=2^k$ and $j=0,1,...m$}
\end{equation}
with a constant $N=N(d, p, T, K, m,K_{\eta}, K_{\xi})$. 
By virtue of this estimate 
there is a subsequence of integers $n'\to\infty$,  such that  
$u_{n'}$ converges weakly 
to some $u$ in each $\mathbb W^j_{p,r}$ for 
$j=0,1,2,...,m$, $p=2^k$ for integers $k\geq1$ and integers $r>1$. 
Due to the weak convergence estimate \eqref{un estimate} remains valid 
also for $u$ with the same constant $N$, i.e. we have 
\begin{equation}                                        \label{estimate u}
|u|^p_{\mathbb W^j_{p,r}}\leq N(|\psi|^p_{W^j_p}+|f|^p_{\mathbb W^j_{p}}), 
\quad \text{for $p=2^k$ and $j=0,1,...m$}. 
\end{equation}

Observe 
that due to the boundedness of $h$ there is a constant $C$, independent of $v\in\mathbb W^1_p$ 
and $n$,  
such that 
\begin{align*}
\Phi_n^2(v)&\leq C
\int_0^T\int_0^1\int_{Z}\int_{\bR^d}
|\eta^k(x)(v(s,x+\theta\eta(x))-v(x))||D_k\varphi(x)|\,dx\,\mu_n(dz)\,d\theta\,ds\\&\leq C
\int_0^T\int_0^1\int_{Z}\int_{\bR^d}\int_0^1\theta
|\eta^k(x)\eta^i(x)D_iv(s,x+\vartheta\theta\eta(x))||D_k\varphi(x)|
\,d\vartheta\,dx\,\mu_n(dz)\,d\theta\,ds\\
&\leq C \int_0^1\theta\int_0^1\int_{Z}{\bar\eta}^2(z)
\int_0^T\int_{\bR^d}|D_iv(s,x+\vartheta\theta\eta(x))||D_k\varphi(x)|
\,dx\,ds\,\mu_n(dz)\,d\vartheta\,d\theta. 
\end{align*}
Hence, taking into account that with a constant $N$
$$
\int_0^T\int_{\bR^d}|D_iv(s,x+\vartheta\theta\eta(x)|^p\,dx\,ds\leq N|v|^p_{\mathbb W^1_p}
$$
for all $v\in \mathbb W^1_p$, by H\"older's inequality we get 
$$
\Phi_n^2(v)\leq N\int_{Z}{\bar\eta}^2(z)\,\mu_n(dz)|v|_{\mathbb W^1_p}|\varphi|_{W^1_q}
$$
with a constant $N$ independent of $v$ and $n$. In the same way we see that 
$$
\Phi^2(v)\leq N\int_{Z}{\bar\eta}^2(z)\,\mu(dz)|v|_{\mathbb W^1_p}|\varphi|_{W^1_q} 
$$
and 
$$
|\Phi^2(v)-\Phi^2_n(v)|
\leq N\int_{Z_n^{c}}{\bar\eta}^2(z)\,\mu(dz)|v|_{\mathbb W^1_p}|\varphi|_{W^1_q},  
$$
where $Z_n^c$ denotes the complement of $Z_n$. 
Hence $\Phi^2, \Phi^2_n\in \mathbb W^{1\ast}_p$ and $\Phi^2_n\to\Phi^2$ 
strongly in $\mathbb W^{1\ast}_p$.  Therefore due to the weak convergence 
$u_{n'}$ to $u$ in $\mathbb W^1_p$, we have 
$$
\lim_{n'\to\infty}\Phi^2_{n'}(u_{n'})=\Phi^2(u). 
$$ 
We can prove in the same way that  
$$
\lim_{n'\to\infty}\Phi^{i}_{n'}(u_{n'})=\Phi^{i}(u)\quad \text{for $i=3,4$}. 
$$
Consequently, letting $n'\to\infty$ in equation \eqref{phin}
we get  
$$
\int_0^Th(t)(u(t),\varphi)\,dt=\int_0^Th(t)(\psi,\varphi)\,dt
$$
$$
+\int_0^Th(t)\int_0^t\langle\cA u(s),\varphi\rangle
+(f(s),\varphi)\,ds\,dt 
$$
for every $\varphi\in C_0^{\infty}(\bR^d)$ and $h\in L_{\infty}([0,T],\bR)$. 

{\it Step 3.} Now we dispense with the additional assumption that $\psi$ and $f$ 
vanish for $|x|\geq R$ for some $R>0$. 
Let  $\psi\in W^m_p$ and $f\in L_p([0,T],W^m_p)$ for $p=2^k$ for some integer 
$k\geq1$. Then for integers $n\geq1$ define 
$\psi^{n}$ and $f^n$ by
$$
\psi^{n}(x)=\psi(x)\chi_n(x),\quad f^n(t,x)
=f(t,x)\chi_n(x),\quad t\in[0,T], \,x\in\bR^d, 
$$
where $\chi_n(\cdot)=\chi(\cdot/n)$ with a nonnegative function 
$\chi\in C_0^{\infty}(\bR^d)$,  
such that $\chi(x)=1$ for $|x|\leq 1$ and $\chi(x)=0$ for $|x|\geq2$. 
Then by virtue of Step 2 equation \eqref{eq1} with $f^n$ in place of $f$ and with initial 
condition $u(0)=\psi^n$ has a solution $u^n$, i.e., 
 $$
\int_0^Th(t)(u^n(t),\varphi)\,dt=\int_0^Th(t)(\psi^n,\varphi)\,dt
$$
\begin{equation}                                                 \label{eq3}
+\int_0^Th(t)\int_0^t\langle\cA u^n(s),\varphi \rangle
+(f^n(s),\varphi)\,ds\,dt 
\end{equation}
for every $\varphi\in C_0^{\infty}(\bR^d)$ and $h\in L_{\infty}([0,T],\bR)$. 
We also have estimate \eqref{estimate u} with $u^n$, $\psi^n$ and $f^n$ 
in place of $u$, $\psi$ and $f$, respectively. Hence for any $n$ and $k$
\begin{equation*}                                                                      
|u^n-u^k|^p_{\mathbb W^j_{p,r}}
\leq N(|\psi^n-\psi^k|^p_{W^j_p}+|f^n-f^k|^p_{\mathbb W^j_{p}})
\end{equation*}
which shows that $u^n$ is a Cauchy sequence in $\mathbb W^j_{p,r}$, and hence 
it converges in the norm of $\mathbb W^j_{p,r}$ to some $u\in\mathbb W^j_{p,r}$ 
for every $j=0,1,2,...,m$ and integers $r>1$. It is easy to pass to the limit in 
equation \eqref{eq3} and see that $u$ solves equation \eqref{eq1} with 
initial and free data $\psi$ and $f$. Clearly, $u$ satisfies 
also the estimate \eqref{estimate u}.

\medskip
Set $\Psi^m_p:=H^m_p$, $\bF^m_p:=L_p([0,T],H_p^m)$  and 
$\mathbb U^m_{p}:=L_{r}([0,T], H^m_p)$ for $m\in[1,\infty)$, $p\in[2,\infty)$ 
and for fixed $r>1$,  and denote by $\bS$ the operator that assigns  
the solution $u$ of equation \eqref{eq1} to $(\psi,f)$, the pair of initial and free data. 
By virtue of Step 3 we know that $\bS$ is a continuous linear operator 
from $\Psi^m_p\times\bF^m_p$ into $\mathbb U^m_{p}$ for 
$p=2^k$, with integers $k\geq1$, for every integer $m\geq1$, 
with operator norm, depending only on $p$, $d$, $T$ and on the constants 
$K$, $K_{\eta}$ and $K_{\xi}$. To show that this holds also for any $p\in[2,\infty)$ 
and any $m\in(1,\infty)$, we use some results from the theory 
of complex interpolation of Banach spaces.  

A pair of complex Banach spaces $A_0$ and $A_1$, which are 
continuously embedded into a Hausdorff  topological vector space $\cH$, is called an interpolation 
couple, and 
$[A_0,A_1]_{\theta}$ denotes the complex interpolation space   
between $A_0$ and $A_1$ with parameter $\theta\in(0,1)$. For an interpolation couple 
$A_0$ and $A_1$ the notation $A_0+A_1$ is used for subspace of 
vectors in $\cH$, 
$\{v_0+v_{1}:v_0\in A_0, \,v_1\in A_1\}$, equipped with the norm 
$$
|v|_{A_0+A_1}:=\inf\{|v_0|_{A_0}+|v_1|_{A_1}
:v=v_0+v_1,v_0\in A_0, \,v_1\in A_1\}. 
$$
Then the following statements hold (see 1.9.3, 1.18.4 and 2.4.2 from \cite{T}).

\begin{enumerate}[(i)]
\item 
If $A_0,A_1$ and $B_0,B_1$ are two interpolation couples and 
$S:A_0+A_1\to B_0+B_1$ is a linear operator such that 
its restriction onto $A_i$ is a continuous operator into $B_i$ with operator norm $C_i$ 
for $i=0,1$, then its restriction onto $A_{\theta}=[A_0,A_1]_{\theta}$ is a continuous operator into 
$B_{\theta}=[B_0,B_1]_{\theta}$ with operator norm $C_0^{1-\theta}C_1^{\theta}$ for every $\theta\in(0,1)$.  
\item 
For a measure space $\mathfrak M$ and $1< p_0,p_1<\infty$, 
$$
[L_{p_0}(\mathfrak M,A_0),L_{p_1}(\mathfrak M,A_1)]_{\theta}=L_{p}(\mathfrak M,[A_0,A_1]_{\theta}),
$$ 
for every $\theta\in(0,1)$, where $1/p=(1-\theta)/p_0+\theta/p_1$.
\item 
For $m_0,m_1\in\bR$, $1<p_0,p_1<\infty$,
$$
[H^{m_0}_{p_0},H^{m_1}_{p_1}]_{\theta}=H^{m}_{p},
$$
where $m=(1-\theta)m_0+\theta m_1$, 
and $1/p=(1-\theta)/p_0+\theta/p_1$.
\item For $\theta\in[0,1]$ there is a constant 
$c_{\theta}$ such that 
$$
|v|_{A_{\theta}}\leq c_{\theta}|v|_{A_0}^{1-\theta}|v|^{\theta}_{A_1}
$$
for all $v\in A_0\cap A_1$. 
\end{enumerate}
Now for an arbitrary $p\geq2$ we take an integer $k\geq1$ and a parameter 
$\theta\in[0,1]$ such that $p_0=2^k\leq p\leq 2^{k+1}=p_1$ 
and $1/p=(1-\theta)/p_0+\theta/p_1$.  
By property (ii) we have
$$
\Psi^m_p=[\Psi^m_{p_0},\Psi^m_{p_1}]_{\theta}=H^m_p, 
\quad
\mathbb{F}^m_p
=[\mathbb{F}^m_{p_0},\mathbb{F}^m_{p_1}]_{\theta}=L_p([0,T],H^m_p),
$$
$$
\mathbb{U}^m_p
=[\mathbb{U}^m_{p_0},\mathbb{U}^m_{p_1}]_{\theta}=L_r([0,T],H^m_p),
$$
and therefore by (i) the solution operator $\mathbb S$ is continuous for any $p\geq 2$  
and integer $m\geq0$.

When $s\in(0,m]$ is not an integer then we set $\theta=s-\lfloor s\rfloor$.  
Then by (ii) and (iii) 
$$
\Psi^s_p=[\Psi^{\lfloor s\rfloor}_{p},\Psi^{\lceil s\rceil}_{p}]_{\theta}=H^s_p,
\quad 
\mathbb{F}^s_p
=[\mathbb{F}^{\lfloor s\rfloor}_{p},\mathbb{F}^{\lceil s\rceil}_{p}]_{\theta}
=L_p([0,T],H^s_p),
$$
$$
\mathbb{U}^s_p
=[\mathbb{U}^{\lfloor s\rfloor}_{p},\mathbb{U}^{\lceil s\rceil}_{p}]_{\theta}
=L_r([0,T],H^s_p)
$$
for every $p\geq2$ and integers $r>1$. 
We have seen above that under the Assumptions \ref{assumption L},  
\ref{assumption M} and \ref{assumption N} with $m\geq1$, 
the solution operator $\bS$ is continuous from 
$\Psi^{\lceil m\rceil}_{p}
\times 
\mathbb{F}^{\lceil m\rceil}_{p}$
to $\mathbb {U}^{\lceil m\rceil}_{p,r}$, 
and from 
$\Psi^{\lfloor m\rfloor}_{p}\times \mathbb{F}^{\lfloor m\rfloor}_{p}$ 
to $\mathbb{U}^{\lfloor m\rfloor}_{p,r}$. 
Hence by (i) again for the solution $u$ we have  
\begin{equation}                                        \label{estimate_ps}
\left(\int_0^T|u(t)|^r_{H^s_{p}}\,dt\right)^{1/r}
\leq N(|\psi|_{H^s_p}+|f|_{\mathbb H^s_{p}})
\end{equation}
with a constant $N=(p,d,m,T,K,K_{\eta},K_{\xi})$. 
Letting here $r\rightarrow\infty$ we obtain 
\begin{equation}                                        \label{spestimate}
\esssup_{t\in[0,T]}|u(t)|_{H^s_{p}}
\leq N(|\psi|_{H^s_p}+|f|_{\mathbb H^s_{p}}). 
\end{equation}

By Lemma \ref{lemma Kr} we already know 
that the solution $u$ is in $C([0,T],H^0_p)$. To show that it is weakly 
continuous as an $H^m_p$-valued function 
we use the following lemma. 
\begin{lemma}                                                                            \label{lemma wcontinuity}
Let $V$ be a reflexive Banach space, 
embedded continuously and densely into a Banach space $U$.  
Let $f$ be a $U$-valued weakly continuous function on $[0,T]$ 
and assume 
there is a dense subset $S$ of $[0,T]$ such that $f(s)\in V$ for $s\in S$ 
and $\sup_{s\in S}|f(s)|_V<\infty$. 
Then $f$ is a $V$-valued function, which is continuous in the weak topology of $V$.
\end{lemma}

\begin{proof}
Since $S$ is dense in $[0,T]$, for a given $t\in[0,T]$ there is a sequence 
$\{t_n\}_{n=1}^{\infty}$ with elements in $S$ such that $t_n\rightarrow t$. 
Due to $\sup_{n\in\N}|f(t_n)|_V<\infty$ and the reflexivity of $V$ there is  a subsequence $\{t_{n_k}\}$ 
such that $f(t_{n_k})$ converges weakly in $V$ to some element $v\in V$. Since $f$ is weakly continuous in $U$, for every continuous linear functional $\varphi$ over $U$ 
we have $\lim_{k\to\infty}\varphi(f(t_{n_k}))=\varphi({f}(t))$. 
Since the restriction of $\varphi$ in $V$ is a continuous 
functional over $V$ we have $\lim_{k\to\infty}\varphi(f(t_{n_k}))=\varphi(v)$. 
Hence $f(t)=v$, which proves that $f$ is a $V$-valued function. Moreover,   
by taking into account that 
$$
|f(t)|_V=|v|_V\leq\liminf_{k\to\infty}|f(t_{n_k})|_V
\leq \sup_{t\in S}|f(t)|_V<\infty,
$$ we obtain $K:=\sup_{t\in[0,T]}|f(s)|_V<\infty$.
Let $\phi$ be a continuous linear functional over $V$. Due to the reflexivity of $V$, the dual $U^*$ 
of the space $U$ is densely embedded into $V^*$, the dual of $V$. Thus for $\phi\in V^{\ast}$ and $\varepsilon>0$ there is 
$\phi_\varepsilon\in U^*$ such that $|\phi-\phi_\varepsilon|_{V^*}\le \varepsilon$. Hence 
$$
|\phi(f(t))-\phi(f(t_n))|\leq |\phi_\varepsilon(f(t)-f(t_n))|+|(\phi-\phi_\varepsilon)(f(t)-f(t_n))|
$$
$$
\leq|\phi_\varepsilon(f(t)-f(t_n))|+\varepsilon|f(t)-f(t_n)|_V
\leq |\phi_\varepsilon(f(t)-f(t_n))|+2\varepsilon K.
$$
Letting here $n\to\infty$ and then $\varepsilon\to0$, we get
$$
\limsup_{n\to\infty}|\phi(f(t))-\phi(f(t_n)) |\leq0,  
$$
which completes the proof of the lemma. 
\end{proof}
Clearly, $u$ is weakly continuous as an $H^0_p$-valued function. 
Hence applying Lemma \ref{lemma wcontinuity} with $V=H^m_p$ and $U=H^0_p$,    
by using \eqref{spestimate} with $s=m$, we obtain that $u$ is weakly continuous 
as an $H^m_p$-valued function. Thus by virtue of \eqref{spestimate} 
we have 
\begin{equation}                                                  \label{sup}
\sup_{t\in[0,T]}|u(t)|_{H^s_{p}}
\leq N(|\psi|_{H^s_p}+|f|_{\mathbb H^s_{p}}) 
\end{equation}
for all $s\in[0,m]$ and $p\geq2$ with a constant $N=N(m,p,d,K,K_{\xi},K_{\eta},T)$.

To show that $u$ is strongly continuous as an $H^s_p$-valued 
function for any $s<m$, notice that by the {\it multiplicative inequality} (iv) we have a constant $c$ 
such that for any sequence $t_n\to t$ in $[0,T]$ we have 
\begin{equation}                                                   \label{ws}
|u(t)-u(t_n)|_{H^s_p}\leq c|u(t)-u(t_n)|^{(m-s)/m}_{L_p}|u(t)-u(t_n)|^{s/m}_{H^m_p}. 
\end{equation}
Letting here $n\to\infty$ we get 
$\lim_{n\to\infty}|u(t)-u(t_n)|_{H^s_p}=0$ by using \eqref{sup} and the strong continuity of $u$ 
as an $L_p$-valued function.  
This shows that $u\in C([0,T],H^s_p)$ for every $s<m$ and finishes  
the proof of Theorem \ref{theorem main} for $V^m_p:=H^m_p$. 

Consider now the case $V^m_p:=W^m_p$. Since for integers $m\geq0$ 
the spaces $H^m_p$ and $W^m_p$ are the same as vector spaces equipped with equivalent 
norms for any $p\geq1$, we need only consider the case when $m$ is not 
an integer and $p\geq2$ is a real number. We will make use of the following facts about the interpolation spaces 
$(A_0,A_1)_{\theta,q}$ with parameters $\theta\in(0,1)$ and $q\in[1,\infty]$, 
obtained by real interpolation methods from an interpolation couple of Banach spaces  
$A_0$ and $A_1$ (see 1.3.3 in \cite{T}).  
\begin{enumerate}[(a)]
\item
If $A_0,A_1$ and $B_1,B_2$ are two interpolation couples and 
$S:A_0+A_1\to B_0+B_1$ is a linear operator such that 
its restriction onto $A_i$ is a continuous operator into $B_i$ with operator norm $C_i$ 
for $i=0,1$, then its restriction onto $A_{\theta,q}=(A_0,A_1)_{\theta,q}$ 
is a continuous operator into 
$B_{\theta,q}=(B_0,B_1)_{\theta,q}$ 
with operator norm $C_0^{1-\theta}C_1^{\theta}$ for every $\theta\in(0,1)$ and 
$q\in[1,\infty]$. 
 \item 
For a measure space $\mathfrak M$ for $p_0,p_1\in(1,\infty)$ we have 
$$
(L_{p_0}(\mathfrak M,A_0),L_{p_1}(\mathfrak M,A_1))_{\theta,p}=L_{p}(\mathfrak M,(A_0,A_1)_{\theta,p}) 
$$ 
for every $\theta\in(0,1)$, where $1/p=(1-\theta)/p_0+\theta/p_1$.
\item
For $s_0,s_1\in(0,\infty)$, $s_0\neq s_1$ 
$$
(W^{s_0}_p,W^{s_1}_p)_{\theta,p}
=W^s_p\quad\text{for $\theta\in(0,1)$ and $p\in(1,\infty)$}  
$$
when $s:=(1-\theta)s_0+\theta s_1$ is not an integer.  
\item 
For $\theta\in(0,1)$ and $q\in[1,\infty]$ there is a constant 
$c_{\theta,q}$ such that 
$$
|v|_{A_{\theta,q}}\leq c_{\theta,q}|v|_{A_0}^{1-\theta}|v|^{\theta}_{A_1}
$$
for all $v\in A_0\cap A_1$. 
\end{enumerate}
\medskip
\noindent
For a fixed $t\in[0,T]$ consider the operator $\bS(t)$ mapping $(\psi,f)\in W^n_p\times L([0,T],W^n_p)$ 
to $u(t)\in W^n_p$, the solution of equation \eqref{eq1} at time $t$. 
We already know that $\bS(t)$ is a bounded 
operator for $p\geq2$ and integers $n\in[0,m]$, and its norm can be estimated 
by the right-hand side of  \eqref{sup} in this case. When $n=s\geq0$ is not an integer, 
then we set $\theta=s-\lfloor s\rfloor$.  Then using (b) and (c) we have 
$$
[W^{\lfloor s\rfloor }_p, W_p^{\lceil s\rceil }]_{\theta,p}=W^s_p, 
\quad 
[L_p([0,T],W^{\lfloor s\rfloor }_p), L_p([0,T],W^{\lceil s\rceil}_p)]_{\theta,p}=L_p([0,T],W^{s}_p), 
$$
and by (a) we get that $u(t)\in W^s_p$ for every $t\in[0,T]$ and $s\in[0,m]$. Moreover, we have  
$$
\sup_{t\in[0,T]}|u(t)|^p_{W^s_p}\leq N|\psi|^p_{W^s_p}+N\int_0^T|f(t)|^p_{W^s_p}\,dt
$$
for every $s\in[0,m]$ and $p\geq2$. Hence taking into account that $u$ is strongly continuous 
in $t$ as an $L_p$-valued function, by (c) we get that it is (strongly) continuous as a $W^s_p$-valued function 
for every $s<m$. Moreover, using Lemma \ref{lemma wcontinuity} with $V=W^m_p$ and $U=L_p$ 
it follows that $u$ is weakly continuous as a $W^m_p$-valued function.

\subsection{Uniqueness of the generalised solution.}
First we assume that $\cR=0$ in equation \eqref{eq1} and 
denote by $\cU$  the set of exponents $p\geq2$ such that 
the statement of Theorem \ref{theorem main} 
on the uniqueness of the solution to equation \eqref{eq1} (with $\cR=0$) 
holds with $p$. Then the following 
proposition completes the proof of the uniqueness for any $p\geq2$. 

\begin{proposition}                                                 \label{proposition uniqueness}
If $q\in\cU$, then $[q,q+d^{-1}]\subset\cU$. 
\end{proposition}
Indeed, since we know $2\in\cU$, a repeated application of this proposition 
gives $\cU=[2,\infty)$, i.e., the uniqueness holds for every $p\geq2$ when $\cR=0$. Now we show 
that hence the uniqueness of the solution 
also in the general case. To this end consider the full equation \eqref{eq1}, 
and let $u_i\in L_p([0,T],W^1_p)=\bW^1_p$  be, for $i=1,2$, generalised solutions to it. Then 
$v=u_1-u_2$ is the unique generalised solution in $\bW^1_p$ to equation \eqref{eq1} with 
$v(0)=0$, 
$\cR=0$ and $f=\cR v$, on any interval $[0,t]$, $t\in[0,T]$. 
Hence by the estimate for the solution we constructed 
in the proof of the existence the solutions, and by 
Assumption \ref{assumption R}, we have a constant $N=N(K,d,T,K_{\eta},K_{\xi})$ such that 
$$
|v(t)|^p_{L_p}\leq N\int_0^t|\cR v(s)|_{L_p}^p\,ds\leq NK^p\int_0^t|v|^p_{L_p}\,dt, \quad t\in[0,T], 
$$
which by Gronwall's lemma implies $|v(t)|_{L_p}=0$ for $t\in[0,T]$. 

To prove the proposition we use for every $R\geq1$ a smooth cutting function 
$\chi_R$, defined by $\chi_R(x)=\chi(x/R)$, $x\in\bR^d$, where 
$\chi$ is a nonnegative smooth function on $\bR^d$ 
such that $\chi(x)=1$ for $|x|\leq1$ and it vanishes for $|x|\geq2$. We introduce also 
 linear operators $F_R=F_R(t)$ and $G_R=G_R(t)$ defined on $W^1_p$ for each 
 $t\in[0,T]$ as follows: 
\begin{equation}                                                                            \label{FG}
F_Rv=\int_ZI^{\eta}\chi_RI^{\eta}v\,\mu(dz), \quad 
G_Rv=\int_ZI^{\xi}\chi_RT^{\xi}v\,\mu(dz), 
\end{equation}
where $T^{\xi}\varphi(x)=\varphi(x+\xi_{t,z}(x))$, $I^{\xi}\varphi(x)=\varphi(x+\xi_{t,z}(x))$ 
and $I^{\eta}\varphi(x)=\varphi(x+\eta_{t,z}(x))$  for $x\in\bR^d$ and each $t\in[0,T]$, 
$z\in Z$, for functions $\varphi$ on $\bR^d$.   

We will prove Proposition \ref{proposition uniqueness} by the following lemma. 
\begin{lemma}                                                      \label{lemma cutting}
Let $2\leq q<p$, $r=d(p-q)/(pq)$. Then there is a constant $N$ such that 
the following estimates hold for all $v\in W^1_p$, $R\geq1$, $t\in[0,T]$ and 
$i,j=1,2,...,d$.  
\begin{enumerate}[(i)]
\item 
$
|vD_i\chi_R|_{L_q}\leq NR^{1-r}|v|_{L_p}$, 
\quad
$|vD_{ij}\chi_R|_{L_q}\leq NR^{r-2}|v|_{L_p}$,
\quad 
$|D_i\chi_R D_jv|\leq NR^{r-1}|v|_{W^1_p}$.  
\item $|v\cM\chi_R|_{L_q}\leq NR^{r-2}|v|_{W^1_p}$, 
\quad
$|F_Rv|_{L_q}\leq NR^{r-1}|v|_{W^1_p}$,
\quad
$|G_Rv|_{L_q}\leq NR^{r-1}|v|_{L_p}$. 
\end{enumerate}
\end{lemma}
\begin{proof} Set $s=qp/(p-q)$. 
Then by the chain rule, H\"older's inequality and by the change of variable 
$y=x/R$ we get 
$$
|vD_i\chi_R|_{L_q}=R^{-1}|v(D_i\chi)(\cdot/R)|_{L_q}
\leq R^{-1}|v|_{L_p}|(D_i\chi)(\cdot/R)|_{L_s}=R^{r-1}|v|_{L_p}|D_i\chi|_{L_s}, 
$$
which proves the first estimate in (i). The other estimates in (i) can be proved in the same way. 
Note that it is enough to prove the estimates in (ii) for smooth $v\in W^1_p$. 
In order to prove the first estimate in (ii) recall the notation $\tau_{\theta\eta}(x):=x+\theta_{t,z}\eta(x)$ and  
notice that due to Assumption \ref{assumption M} 
by Taylor's formula we have 
$$
|\cM\chi_R|
\leq R^{-2}\int_{Z}
\int_0^1\bar\eta^2(z)|(D^2\chi)(\tau_{\theta\eta}(x)/R)|\,d\theta\,\mu(dz)
$$
for $t\in[0,T]$ and $z\in Z$. Hence by the Minkowski and H\"older inequalities,  
then by the change of 
variable $y=\tau_{\theta\eta}(x)/R$ and using Assumption \ref{assumption M} we get 
$$
|v\cM\chi_R|_{L_q}\leq R^{-2}\int_{Z}
\int_0^1\bar\eta^2(z)|v|_{L_p}|(D^2\chi)(\tau_{\theta\eta}(x)/R)|_{L_s}\,d\theta\,\mu(dz)
$$
$$
\leq K^2_{\eta}K^{1/s}R^{r-2}|v|_{L_p}|D^2\chi|_{L_s}, 
$$
which proves the first estimate in (ii). To prove the second estimate in (ii) we use 
Taylor's formula and Assumption \ref{assumption M} to get 
$$
|F_Rv|\leq R^{-2}
\int_Z\int_0^1\int_0^1\bar\eta^2(z)
|(D\chi)(\tau_{\theta\eta}(x)/R)||(Dv)(\tau_{\vartheta\eta}(x))|\,d\theta\,d\vartheta\,\mu(dz). 
$$
for every $z\in Z$, $t\in[0,T]$ and $x\in\bR^d$. 
Hence by the Minkowski and the H\"older inequalities and then by changes of variables 
$y=\tau_{\theta\eta}(x)/R$ and $y'=\tau_{\vartheta\eta}(x)$ we obtain 
$$
|F_Rv|_{L_q}\leq R^{-2}
\int_Z\int_0^1\int_0^1\bar\eta^2(z)
|(D\chi)(\tau_{\theta\eta}/R)(Dv)(\tau_{\vartheta\eta})|_{L_q}\,d\theta\,d\vartheta\,\mu(dz) 
$$
$$
\leq \int_Z\int_0^1\int_0^1\bar\eta^2(z)|(Dv)(\tau_{\vartheta\eta})|_{L_p}
|(D\chi)(\tau_{\theta\eta}/R)|_{L_s}\,d\theta\,d\vartheta\,\mu(dz)
$$
$$
\leq K^2_{\eta}K^{1/p}K^{1/s}R^{r-2}|Dv|_{L_p}|D\chi|_{L_s}, 
$$
which proves the second estimate in (ii). Finally, by Taylor's formula
$$
|G_Rv|\leq \int_{Z}\int_0^1\bar \xi(z)|
v(\tau_{\xi})||D\chi(\tau_{\theta\xi}/R)|\,d\theta\,\mu(dz)
$$
for every $(t,z,x)$, where recall that $\tau_{\theta\xi}=x+\theta\xi_{t,z}(x)$ and 
$\tau_{\xi}=x+\xi_{t,z}(x)$.  
Hence by the Minkowski and the H\"older inequalities and then by changes of variables 
$y=x+\xi_{t,z}(x)$ and $y'=(x+\theta\xi(x))/R$ we obtain 
$$
|G_Rv|_{L_q}\leq R^{-1}\int_{Z}\int_0^1\bar \xi(z)
|v(\tau_{\xi})|(D\chi_R)(\tau_{\theta\xi})||_{L_q}\,d\theta\,\mu(dz)
$$
$$
\leq R^{-1}\int_{Z}\int_0^1\bar \xi(z)
|v(\tau_{\xi})|_{L_p}|(D\chi_R)(\tau_{\theta\xi})|_{L_s}\,d\theta\,\mu(dz)
$$
$$
\leq K^{1/p}K^{1/s}K_{\xi}R^{r-1}|D\chi|_{L_s}, 
$$
which proves the last estimate in (ii).
\end{proof}
\begin{proof}[Proof of Proposition \ref{proposition uniqueness}]
 Let $q\in\cU$, $p\in[q,q+1/d]$, assume that Assumptions \ref{assumption L} 
 through \ref{assumption R} hold with $p$, 
 and let $u_i\in \bW^1_p$ be generalised solution 
 to equation \eqref{eq1} (with $\cR=0$) with initial condition $u_i(0)=\psi$ 
 for $i=1,2$. Then $v:=u_1-u_2$ is a generalised solution of \eqref{eq1} 
 (with $\cR=0$), $v(0)=0$ and $f=0$. 
 We want to get an equation for $v_R:=v\chi_R$. To this end notice first that 
 for $w\in W^2_p$ we have 
 $$
 \chi_R\cL w
 =\cL(\chi_Rw)-b_iwD_i\chi_R-a^{ij}wD_{ij}\chi-2a^{ij}D_iwD_j\chi_R, 
 $$
 \begin{equation}                                                         \label{MR}
 \chi_R\cM w=\cM(\chi_Rw)-w\cM\chi_R-F_Rw, \quad 
 \chi_R\cN w=\cN(\chi_Rw)-G_Rw, 
 \end{equation}
 where the operators $F_R$ and $G_R$  are defined in \eqref{FG}. 
 Clearly, $v_R\in W^1_q$. Hence, by using test functions $\chi_R\varphi$ instead of 
$\varphi$  in \eqref{solution},  it is not difficult to see that $v_{R}$ is a generalised 
solution of 
$$
dv_R(t)=(\cA v_R(t)+f_R(t))\,dt, \quad v_R(0)=0, 
$$
with 
$$
f_R=-b_ivD_i\chi_R-a^{ij}vD_{ij}\chi-2a^{ij}D_ivD_j\chi_R
-v\cM\chi_R-F_Rv-G_Rv,  
$$
which by the previous lemma belongs to $\bL_q=L_q([0,T],L_q)$. 
Since $q\in\cU$, this solution is unique. Consequently, we can apply the estimate 
we have for the solutions constructed in the existence proof, according to which 
we have 
$$
|v_R|_{\bL_q}\leq N|f_R|_{\bL_q} 
$$
with a constant $N$ independent of $R$. Using Lemma \ref{lemma cutting} it is easy to 
see that there is a constant $N'$ such that 
$$
|f_R|_{\bL_q}\leq N'R^{r-1}|v|_{\bL_p},  
$$
for all $R\geq1$. Hence 
\begin{equation}                                                      \label{eqR}
|v_R|_{\bL_q}\leq NN'R^{r-1}|v|_{\bL_p}
\end{equation}
for all $R\geq1$. Notice that for $p\in[q, q+d^{-1}]$ we have 
$$
r-1=\frac{dp-dq-pq}{pq}\leq \frac{1-q^2}{pq}<0. 
$$
Consequently, letting $R\to\infty$ in \eqref{eqR} we get $|v|_{\bL_q}=0$, 
which finishes the proof of the proposition. 
\end{proof}

\smallskip
\noindent 
 {\bf Acknowledgement.} A generalisation of Theorem \ref{theorem main} 
 to stochastic integro-differential equations was presented at the 
 conference on ``Harmonic Analysis for Stochastic PDEs" in Delft, 10-13 July, 2018 and 
 at  the ``9th International Conference
on Stochastic Analysis and Its Applications" in Bielefeld, 3-7 September, 2018. 
 The authors are grateful to the organisers of these conferences for the invitation 
and for discussions.  
They also thank Alexander 
 Davie in Edinburgh University for correcting some mistakes.

\end{document}